\documentclass[a4paper]{article}
\usepackage[T1]{fontenc}
\usepackage{lmodern}
\usepackage[english]{babel}
\usepackage[utf8]{inputenc}
\usepackage{setspace}
\usepackage{hyperref}
\usepackage{microtype}
\usepackage{quoting}
\usepackage{mathrsfs}
\usepackage{booktabs}
\usepackage{caption}
\usepackage{subfig}
\usepackage{graphicx}
\usepackage{multirow}
\usepackage{yhmath}
\usepackage{xcolor}

\usepackage{amssymb}
\usepackage{amsthm}
\usepackage{amsmath}
\usepackage[linesnumbered]{algorithm2e}
\SetKwInput{init}{Initialization}

\newtheorem{lemma}{Lemma}
\newtheorem{theorem}{Theorem}
\newtheorem{assumption}{Assumption}
\newtheorem{definition}{Definition}
\newtheorem{corollary}{Corollary}
\newtheorem{remark}{Remark}

\newcommand{\edit}{\color{black}}
\newcommand{\editSte}{\color{black}}		

\begin{document}

\title{A regularized Interior Point Method for sparse Optimal Transport on Graphs}

\author{Stefano Cipolla\footnote{School of Mathematics, University of Edinburgh, Edinburgh, UK. \href{mailto:scipolla@exseed.ed.ac.uk}{scipolla@exseed.ed.ac.uk}}\and Jacek Gondzio\footnote{School of Mathematics, University of Edinburgh, Edinburgh, UK. \href{mailto:j.gondzio@ed.ac.uk}{j.gondzio@ed.ac.uk}} \and Filippo Zanetti\footnote{School of Mathematics, University of Edinburgh, Edinburgh, UK. \href{mailto:f.zanetti@sms.ed.ac.uk}{f.zanetti@sms.ed.ac.uk}}}
\date{}
\maketitle

\begin{abstract}
\noindent 
In this work, the authors address the Optimal Transport (OT) problem on graphs using a proximal stabilized Interior Point Method (IPM). In particular, strongly leveraging on the induced primal-dual regularization, the authors propose to solve
large scale OT problems on sparse graphs using a bespoke IPM algorithm able to suitably exploit  primal-dual  regularization in order to enforce scalability. Indeed, the authors prove that the introduction of the regularization allows to use sparsified versions of the normal Newton equations to inexpensively generate IPM search directions. A detailed theoretical analysis is carried out showing the polynomial convergence of the inner algorithm in the proposed
computational framework. Moreover, the presented numerical results showcase the efficiency and robustness of the proposed approach when compared to network simplex solvers. 
\end{abstract}


\section{Introduction} \label{sec:Intro}
The Optimal Transport (OT) problem requires to move a certain distribution of mass from one configuration into another, minimizing the total cost required for the operation. It has been studied extensively, from the early work of Kantorovich \cite{Kan:OT}, to the development of ever faster algorithms for various OT formulations, e.g.\ \cite{Cut:sinkhorn,GotSch:shortlist,LinOka:OT,Que:AHA2,PeyCut:computationalOT,SchSchGot:dotmark}. Recently, there has been a growing interest in using Interior Point Methods (IPMs) \cite{MR2881732} in applications that involve optimal transport, in particular for very large scale instances of such problems, see e.g.\ \cite{NatTod:OT,WijChe:mfipm,ZanGon:OT}. 

A particularly interesting problem is the optimal transport over {\it sparse} graphs: in this case, the transport of mass is only possible along a specific subset of connections, which is noticeably smaller than the full list of edges of a fully connected bipartite graph, as it would happen in a standard discrete OT formulation. 
{\edit
The use of OT and the Wasserstein distance (i.e.\ the optimal objective function of the OT problem) is becoming more and more common in many practical applications, e.g. neural networks \cite{GulAhmArjMunCou}, image processing \cite{HakZhuTanAng}, inverse problems \cite{MetBroMerOudVir} and in the analysis of large complex networks \cite{Community_detection}.
}

The specific formulation of the problem is the following: suppose that $G = (V,E)$ is a connected graph with directed edges $E \subset V \times V$ and weights $\mathbf c \in \mathbb{R}_+^{|E|}$. Define the incidence matrix $A \in \{-1,0,1\}^{|V| \times |E|}$ as
\begin{equation*}
A_{ve}:=\begin{cases}
		-1, & \hbox{ if }\ e=(v,w) \hbox{ for some } w \in V  \\
		 1, & \hbox{ if }\ e=(w,v) \hbox{ for some } w \in V  \\
		 0, & \hbox{ otherwise. }
	\end{cases}
\end{equation*}
We consider the optimal transport problem in the \textit{Beckmann} form \cite{MR3820384}:
\begin{equation}
\label{eqn:problem_formulation}
	\mathcal{W}_1(\boldsymbol{\rho_0},\boldsymbol{\rho_1}):=\begin{cases}
		\min_{\mathbf{x} \in \mathbb{R}^{|E|}} & \sum_{e \in E} c_e\mathbf{x}_e \\
		s.t.                     & A\mathbf{x}=\boldsymbol{\rho_1}-\boldsymbol{\rho_0}\\   
		                         & \mathbf{x} \geq 0
	                                  
	\end{cases},
\end{equation}
where $\boldsymbol{\rho_0}, \boldsymbol{\rho_1} \in \{\boldsymbol{\rho} \in \mathbb{R}^{|V|} \,:\, \mathbf{1}^T\boldsymbol{\rho}=1 \hbox{ and } \boldsymbol{\rho} \geq 0 \}=: Prob(V) $. In the following we will define $|E|:=n$ and $|V|:=m$. OT on graphs has been recently studied in \cite{MR3820384,FacBen:OT} and, in this formulation,  it is similar to the more general minimum cost flow problem on networks \cite{AhuMagOrl:networkflows}, which has also seen extensive use of IPMs, e.g.\ \cite{CasNas:ipm,FraGen:kkt,MR1740364,ResVei:networks}. 

Sparse graphs have on average very few edges per node, which can lead to nearly disconnected regions and seriously limit the possible paths where mass can be moved. As a result, finding a solution to the optimal transport problem on a sparse graph requires more sophisticated algorithms and may be more computationally challenging compared to solving the same problem on a denser graph. In particular, first order methods like the network simplex may struggle and move slowly towards optimality, due to the limited number of edges available, while an interior point method manages to identify quickly the subset of basic variables (i.e.\ the subset of edges with non-zero flow) and converges faster.

In this work, the authors address the efficient solution of the optimal transport problem \eqref{eqn:problem_formulation} considering the Proximal-Stabilized  Interior Point framework (PS-IPM), recently introduced and analysed in \cite{Cipolla_Gondzio}.

As originally observed in \cite{MR1740364},  when IPMs are used to solve the minimum cost flow problem on networks, the normal form of the related Newton systems is structured as a Laplacian matrix of the graph {\editSte (defined as the difference of the diagonal matrix of the vertex degrees minus the adjacency matrix)} and the iterates of IPM determine the associate weights of this matrix{\editSte, see also eq. \eqref{eq:Laplacian_IPM} }.
In \cite{10.1145/1374376.1374441}, this observation was exploited to solve such Laplacian linear systems (which are, in turn, particular instances of symmetric M-matrices) through the fast specialized solution of $O(\ln m)$ linear systems  involving symmetric diagonally dominant matrices \cite{MR3228466}. We refer the interested reader to \cite{MR3071502} for a survey on fast Laplacian solvers and to \cite{MR1871316} for information concerning the distribution of Laplacian's singular values.

\subsection{Contribution and organization}

This work focuses on the efficient solution of large scale OT problems on sparse graphs using a bespoke IPM algorithm able to suitably exploit  primal-dual  regularization in order to enforce scalability. The organization of the work and its main contributions can be summarized as follows:
\begin{itemize}
	
\item In Section \ref{sec:computational framework}, the authors briefly recall the proximal stabilized framework responsible for the primal-dual  regularization of the IPMs here considered.

\item In Section \ref{sec:convergence}, the authors provide a detailed convergence analysis of the inexact infeasible primal-dual regularized IPM, when a proximal stabilization procedure is used. Moreover, they prove  its polynomial complexity. 

\item In Section \ref{sec;properties_regularized_normal}, the authors prove that the normal form of the related Newton system is naturally structured as a shifted Laplacian matrix characterized by a strict diagonal dominance. Such feature consistently simplifies the factorization of the normal equations and  allows the use of standard libraries for the solution of the corresponding linear systems.  On the other hand, such factorizations could incur a significant fill-in even when the original graph is sparse, hence limiting the applicability of the proposed approach for the solution of large scale problems.

\item In Section \ref{sec:Sparsification}, to overcome potential scalability issues related to the fill-in mentioned above, the authors propose to generate IPM search directions using {\it sparsified} versions of the IPM normal equations. {\edit In particular, the original normal matrix takes the form $A(\Theta^{-1}+\rho I)^{-1}A^T+\delta I$, where $\rho,\delta$ are regularization parameters and $\Theta$ is a diagonal matrix related to the IPM barrier parameter $\mu$; the authors propose to use a \textit{perturbed} normal matrix, where the entries of $(\Theta^{-1}+\rho I)^{-1}$ that are sufficiently small (when compared to $\mu$) are set to zero (completely ignoring the corresponding columns of matrix $A$). This strategy reduces the time required to assemble and solve the normal equations systems, providing a fundamental advantage to the algorithm.}

%which have a natural interpretation in terms of detected optimal transportation plan. Indeed, in the approach here presented, the IPM  Newton directions are generated using a \textit{perturbed} normal matrix obtained by taking into account only the variables recognized as ``basic'' by the IPM procedure.

The resulting sparsified linear systems are solved either using a Cholesky factorization 
(if that displays only negligible fill-in) or using the conjugate gradient method and employing a simple and inexpensive  incomplete Cholesky preconditioner. 
In both these cases either the {\it complete}  or the {\it incomplete} Cholesky factorization
remains very sparse, and this translates into an outstanding efficiency of the proposed method. 
Moreover, the authors are able to interpret the Newton directions generated using sparsified Newton matrices as {\it inexact} Newton directions. Relying on the convergence theory developed in Section \ref{sec:convergence}, the authors are able to prove that, under suitable choice of the sparsification parameters, the above described approach gives rise to a polynomially convergent algorithm.

\item In Section \ref{sec:numer_res}, the authors present experimental results which demonstrate the efficiency and robustness of the proposed approach. Extensive numerical experiments, involving very large and sparse graphs coming from public domain random generators as well as from real world applications, show that, for sufficiently large problems, the approach  presented in this work consistently outperforms, in terms of computational time, the Lemon network simplex implementation \cite{lemon_paper}, one of the state-of-the-art solvers available for network problems. 

\end{itemize}

\subsection{Notation}
In the paper, vectors are indicated with bold letters. $\|\cdot\|$ indicates the Euclidean norm. $I$ represents the identity matrix and $\mathbf e$ the vector of all ones. Given a vertex $v$ of a graph $G$, we denote as $deg(v)$ its degree, i.e.\ the number of edges that are incident to $v$. Concerning the variables inside the algorithm, we use a subscript $k$ to indicate the external proximal iteration and a superscript $j$ to indicate the internal IPM iteration. Given a sequence $\{\mu^j\}_{j \in \mathbb{N}}$ and a continuos function $f$, the big-O notation $O(\cdot)$ is used as follows:
\begin{equation*}
	\{u^j\}_{j \in \mathbb{N}} \in O\big(f(\mu^j)\big) \hbox{ iff }  \exists \; \; C>0   \hbox{ s.t. }  u^j \leq C f(\mu^j) \hbox{ for all } j \in \mathbb{N}. 
\end{equation*}

%\subsection{Motivations}
%
%Interior Point Methods \cite{MR2881732} are fast and we like them very much.
%
%\subsection{The Problem}
%
%In this work we consider the optimal transport problem on graphs.  Suppose that $G = (V,E)$ is a connected graph with directed edges $E \subset V \times V$ and weights $c \in \mathbb{R}_+^{|E|}$. Define the incidence matrix $A \in \{-1,0,1\}^{|V| \times |E|}$ as
%\begin{equation*}
%A_{ve}:=\begin{cases}
%		-1, & \hbox{ if }\ e=(v,w) \hbox{ for some } w \in V  \\
%		 1, & \hbox{ if }\ e=(w,v) \hbox{ for some } w \in V  \\
%		 0, & \hbox{ otherwise. }
%	\end{cases}
%\end{equation*}
%
%We consider the Optimal Transport Problem in the \textit{Beckmann Problem} form \cite{MR3820384}:
%
%\begin{equation}
%	\mathcal{W}_1(\boldsymbol{\rho_0},\boldsymbol{\rho_1}):=\begin{cases}
%		\min_{\mathbf{x} \in \mathbb{R}^{|E|}} & \sum_{e \in E} c_e\mathbf{x}_e \\
%		s.t.                          & \mathbf{x} \geq 0 \\
%	                                  & A\mathbf{x}=\boldsymbol{\rho_1}-\boldsymbol{\rho_0}
%	\end{cases}
%\end{equation}
%where $\boldsymbol{\rho_0}, \boldsymbol{\rho_1} \in \{\boldsymbol{\rho} \in \mathbb{R}^{|V|} \,:\, \mathbf{1}^T\boldsymbol{\rho}=1 \hbox{ and } \boldsymbol{\rho} \geq 0 \}=: Prob(V) $. In the following we will define $|E|:=n$ and $|V|:=m$.

\section{Computational Framework} \label{sec:computational framework}

\subsection{Proximal-Stabilized Interior Point Method}

Let us consider the following primal-dual formulation of a Linear Program (LP):

\begin{align} \label{eq:LP_problem}
	\begin{aligned}
		\min_{\mathbf{x} \in \mathbb{R}^n} \;& \mathbf{c}^T\mathbf{x}  \\ 
		\hbox{s.t.} \; & A\mathbf{x}= \mathbf{b} \\
		& \mathbf{x}  \geq 0\\
	\end{aligned} \;\; \; \; \;\; \; \;
	\begin{aligned}
		\max_{\mathbf{s}  \in \mathbb{R}^n, \; \mathbf{y} \in \mathbb{R}^m} \;& \mathbf{b}^T\mathbf{y}\\ 
		\hbox{s.t.} \; & 
		\mathbf{c}-A^T\mathbf{y}-\mathbf{s}=0 \\
		& \mathbf{s}\geq 0\\
	\end{aligned}
\end{align}
where $A \in \mathbb{R}^{m \times n}$ with $m \leq n$ is not required to have full rank.
{\edit Notice that problem \eqref{eqn:problem_formulation} is indeed formulated in this way.

We solve this problem using PS-IPM \cite{Cipolla_Gondzio}, which is a \textit{proximal-stabilized} version of classic \textit{Interior Point Method}.} {\editSte Broadly speaking, PS-IPM resorts to the Proximal Point Method (PPM) \cite{rockafellar1970monotone} to produce primal-dual regularized forms of problem \eqref{eq:LP_problem}. Indeed, given an approximation $(\mathbf{x}_k,\mathbf{y}_k)$ of the solution of such problem, PS-IPM uses interior point methods  to produce the next PPM step $(\mathbf{x}_{k+1},\mathbf{y}_{k+1})$, which, in turn, represents a \textit{better} approximation of the solution  of problem \eqref{eq:LP_problem}. 

In this regard, the problem that needs to be solved at every PPM step takes the form}
\begin{align} \label{eq:LP_r_problem}
	\begin{aligned}
		\min_{\substack{\mathbf{x} \in \mathbb{R}^n\\ \mathbf{y} \in \mathbb{R}^m}} \;& \mathbf{c}^T\mathbf{x}+\frac{\rho}{2}\|\mathbf{x}-\mathbf{x}_k\|^2 +\frac{\delta}{2}\|\mathbf{y}\|^2  \\ 
		\hbox{s.t.} \; & A\mathbf{x}+\delta(\mathbf{y}-\mathbf{y}_k)= \mathbf{b} \\
		& \mathbf{x} \geq 0, \\
	\end{aligned} 
	\qquad
	\begin{aligned}
		\max_{\substack{\mathbf{x}, \; \mathbf{s} \in \mathbb{R}^n\\ \mathbf{y} \in \mathbb{R}^m}} \;& \mathbf{y}^T\mathbf{b}   {-\frac{\rho}{2}\|\mathbf{x}\|^2-\frac{\delta}{2}\|\mathbf{y}-\mathbf{y}_k\|^2} \\ 
		\hbox{s.t.} \; & 
		\rho(\mathbf{x}- \mathbf{x}_k)-A^T\mathbf{y} - \mathbf{s}+\mathbf{c}  =0 \\
		& \mathbf{s}  \geq 0 \\
	\end{aligned}. \tag{PPM$(k)$}
\end{align} 

{\editSte
\begin{definition}{Solution of problem \eqref{eq:LP_r_problem} \\}
	Using standard duality theory, we say that $(\mathbf{x}_k^*,\mathbf{y}_k^*,\mathbf{s}_k^*)$ is a solution of problem \eqref{eq:LP_r_problem} if the following identities hold
	\begin{align}\label{eqn:optimal_solution_PPM}
		        A\mathbf{x}_k^* +\delta(\mathbf{y}_k^* -\mathbf{y}_k) - \mathbf{b} =0 \notag\\
		        \rho(\mathbf{x} - \mathbf{x}_k)-A^T\mathbf{y}_k^* - \mathbf{s} +\mathbf{c}=0  \\
		            (\mathbf{x}_k^*)^T\mathbf{s}_k^*=0  \hbox{ and }  (\mathbf{x}_k^*,\mathbf{s}_k^*) \geq 0  \notag
	\end{align}
\end{definition}
}

{\editSte
More in particular, the PS-IPM here considered uses two nested cycles to solve problem \eqref{eq:LP_problem}. 
The outer loop uses an inexact proximal point method \cite{MR732428}, as shown in Algorithm~\ref{alg:PS-MF-IPM}: the current approximate solution $(\mathbf{x}_k,\mathbf{y}_k)$ is used to regularize the LP problem, which is then solved using an IPM to find the next approximate solution $(\mathbf{x}_{k+1},\mathbf{y}_{k+1}) \approx (\mathbf{x}^*_{k},\mathbf{y}^*_{k}) $. And indeed, at the inner loop level, an inexact infeasible interior point method
is used to solve the PPM sub-problems, see Algorithm~\ref{alg:IPM}.
{\editSte
Notice that both methods are {\it inexact}: the outer cycle is inexact because the sub-problems are solved approximately by an IPM; the IPM is inexact because the Newton systems are also solved inexactly (see Section \ref{sec:inexact_IPM} for more details). Notice also that the IPM is referred to as {\it infeasible} because the intermediate iterates are not required to be inside the feasible region. We also call the inner loop {\it regularized}, because it is a primal-dual regularized version of the original LP \eqref{eq:LP_problem}.} 

Regularization in interior point methods was originally introduced in \cite{S&T} and extensively used in \cite{MR1777460}, as a tool to stabilize and improve the linear algebra routines needed for their efficient implementation. In this work and in \cite{Cipolla_Gondzio}, the regularization is introduced as a result of the application of the PPM at the outer cycle level. To summarize, in the following we use three acronyms: PPM refers to the outer cycle; IPM refers to the inner cycle; PS-IPM refers to the overall procedure, combining PPM and IPM.
}

\begin{algorithm}[hbt!]
	\caption{PPM, outer loop of PS-IPM}\label{alg:PS-MF-IPM}
	\KwIn{ $tol > 0$, $\sigma_r \in (0,1)$, $\tau_1>0$.  }
	\init{Iteration counter $k = 0$; initial point $(\mathbf{x}_0,\mathbf{y}_0)$}
	\While{Stopping Criterion \eqref{eqn:Alg1_stop} False}{
		Use Algorithm \ref{alg:IPM} with starting point $(\mathbf{x}^0_{k}, \mathbf{y}^0_{k})=(\mathbf{x}_{k}, \mathbf{y}_{k})$ to find $(\mathbf{x}_{k+1},\mathbf{y}_{k+1})$ s.t. 
		\begin{equation}\label{eq:stopping_condition}
			\|\mathbf{r}_k(\mathbf{x}_{k+1},\mathbf{y}_{k+1})\| < \frac{{\sigma_r^k}}{\tau_1} \min\{1, \|(\mathbf{x}_{k+1}, \mathbf{y}_{k+1})-(\mathbf{x}_{k}, \mathbf{y}_{k}) \|\}
		\end{equation} \\
		
		%		$\|r_k(\mathbf{x}_{k+1},\mathbf{y}_{k+1}))\| < \frac{\min({\rho,\delta})}{\tau_1}\sigma_r^k \min\{1, \|(\mathbf{x}_{k+1}, \mathbf{y}_{k+1})-(\mathbf{x}_{k}, \mathbf{y}_{k}) \|\} $ \\
		Update the iteration counter: $k := k + 1$.
	}
\end{algorithm}

{\editSte
Concerning the stopping criteria, we finally highlight that Algorithm~\ref{alg:PS-MF-IPM} is stopped based on the criterion \eqref{eqn:Alg1_stop}. Algorithm~\ref{alg:IPM} instead, is stopped according to the accuracy that is required for the solution of current sub-problem and based on the following \textit{natural residual}, see \cite{MR732428}, of problem \eqref{eq:LP_r_problem}:}

\begin{definition}[Natural Residual]
	\begin{equation*} \label{eq:PPM_res_reg}
		\mathbf{r}_k(\textbf{x},\mathbf{y}):=\begin{bmatrix}
			\mathbf{x} \\
			\mathbf{y}
		\end{bmatrix} - \Pi_{D} \Big ( \begin{bmatrix}
			\mathbf{x} \\
			\mathbf{y}
		\end{bmatrix}- \begin{bmatrix}
			\rho  ( \mathbf{x}-\mathbf{x}_k) +\mathbf{c}-A^T\mathbf{y} \\
			A\mathbf{x}-\mathbf{b}+ \delta (\mathbf{y}-\mathbf{y}_k) 
		\end{bmatrix} \Big ),
	\end{equation*} 
where $$D:=\mathbb{R}_{\geq 0}^{n} \times \mathbb{R}^{m}$$ and where $\Pi_D$ is the corresponding projection operator. {\editSte Moreover,  it is easy to verify that $(\mathbf{x}_k^*,\mathbf{y}_k^*,\mathbf{s}_k^*)$ is a solution of problem \eqref{eq:LP_r_problem} if and only if $\mathbf{r}_k(\mathbf{x}_k^*,\mathbf{y}_k^*)=0$, \cite[Sec. 2.3]{Cipolla_Gondzio}. }
\end{definition}

{\edit
\subsection{Interior point method} \label{sec:inexact_IPM}

We now focus on the inner cycle and give a brief description of the IPM used to solve problem \eqref{eq:LP_r_problem}.  {\editSte To this aim, we introduce the following Lagrangian function which uses a logarithmic barrier  to take into account the inequality constraints
	\begin{equation}\label{eq:AR_Lagragian}
		\begin{split}
			L_k(\mathbf{x}, \mathbf{y})=&\frac{1}{2}[\mathbf{x}^T, \mathbf{y}^T] \begin{bmatrix}
				\rho I & 0 \\
				0 & \delta I
			\end{bmatrix}\begin{bmatrix}
				\mathbf{x} \\
				\mathbf{y}
			\end{bmatrix} +[\mathbf{c}^T- \rho \mathbf{x}_k^T, 0 ]\begin{bmatrix}
				\mathbf{x} \\
				\mathbf{y}
			\end{bmatrix} \\
			&-\mathbf{y}^T(A
			\mathbf{x} + \delta (\mathbf{y}-\mathbf{y}_k) -\mathbf{b}) - \mu \sum_{i =1}^{n} \ln (x_i).
		\end{split}
	\end{equation}
}

\noindent The KKT conditions that arise from the gradients of the Lagrangian \eqref{eq:AR_Lagragian} are
}
\begin{align*}
	\nabla_{\mathbf{x}}L_k(\mathbf{x},\mathbf{y})=	 \rho\mathbf{x}-A^T\mathbf{y}+\mathbf{c} -\rho{\mathbf{x}_k} - \begin{bmatrix}
		\frac{\mu}{x_{1}} \\
		\vdots \\
		\frac{\mu}{x_{n}}
	\end{bmatrix} =0 ; \label{eq:dual_R_feasibilityKKT1} \\
	-\nabla_{\mathbf{y}}L_k(\mathbf{x}, \mathbf{y}) = (A\mathbf{x}+\delta (\mathbf{y}-\mathbf{y}_k) -\mathbf{b})=0 . %\label{eq:primal_R_KKT_1}.
\end{align*}

%Regularization in Interior Point Methods was originally introduced in \cite{S&T} and extensively used in \cite{MR1777460}, as a tool to stabilize and improve the linear algebra routines needed for their efficient implementation. In this work we consider a regularized version of IPMs which leverages on the Proximal Point Method  to have theoretical guarantee of convergence. In this section, for the sake of readability,  we will briefly recall its main properties specializing the details for the Linear Programming (LP) case considered in this work. 

%The Proximal-Stabilized Interior Point Method, see \cite{Cipolla_Gondzio}, uses the Inexact Proximal Point Method as outer scheme, see Algorithm \ref{alg:PS-MF-IPM}, and a standard IPM as an inner scheme, see Algorithm \ref{alg:IPM}.

\noindent Setting $s_i = \frac{\mu}{x_i}$ for $i \in \{1, \dots, n\}$, we consider the following function

\begin{equation}\label{eq:KKT_map}
	F_k^{\mu, \sigma}(\mathbf{x},\mathbf{y}, \mathbf{s}):=\begin{bmatrix}
		\rho  (\mathbf{x}-{\mathbf{x}_k})-A^T \mathbf{y} - \mathbf{s} +\mathbf{c}\\
		A\mathbf{x}+\delta (\mathbf{y}-\mathbf{y}_k) -\mathbf{b}
		\\
		SX\mathbf{e}-\sigma \mu \mathbf{e}
	\end{bmatrix},
\end{equation}
{where $\sigma \in (0,1)$ is the barrier reduction parameter, $S=\text{diag}(\mathbf s)$ and $X=\text{diag}(\mathbf x)$.} A primal–dual interior point method  applied to problem \eqref{eq:LP_r_problem} relies on the use of  Newton iterations to solve a nonlinear problem of the form
\begin{equation*}
	F_k^{\mu, \sigma }(\mathbf{x},\mathbf{y}, \mathbf{s})=0, \; \; \mathbf{x}, \; \mathbf{s}>0.
\end{equation*}
A Newton step for \eqref{eq:KKT_map} from the current iterate $(\mathbf{x},\mathbf{y}, \mathbf{s})$ is obtained by solving the system
  \begin{equation} \label{eq:Newton_System}
  	\begin{bmatrix}
  		\rho I & -A^T & -I \\
  		A    & \delta I & 0 \\
  		S      &  0       & X
  	\end{bmatrix}\begin{bmatrix}
  		\Delta \mathbf{x} \\
  		\Delta \mathbf{y} \\
  		\Delta \mathbf{s}
  	\end{bmatrix} =-F_k^{\mu,\sigma}(\mathbf{x},\mathbf{y},\mathbf{s})=:\begin{bmatrix}
  	\boldsymbol{\xi}_d \\
  	\boldsymbol{\xi}_p \\
  	\boldsymbol{\xi}_{\mu,\sigma}
  \end{bmatrix},
  \end{equation}
i.e., the following relations hold:
\begin{align} 
	& \rho \Delta \mathbf{x} - A^T \Delta \mathbf{y} - \Delta \mathbf{s} = \boldsymbol{\xi}_d  \label{eq:Newton_details} \\
	& A\Delta \mathbf{x} + \delta \Delta \mathbf{y} =  \boldsymbol{\xi}_p \label{eq:normal_sol_2} \\
	& S \Delta \mathbf{x} + X \Delta \mathbf{s} = \boldsymbol{\xi}_{\mu,\sigma} \label{eq:normal_sol_3},
\end{align}  
{\edit
where $(\Delta\mathbf x,\Delta\mathbf y,\Delta\mathbf s)$ is the Newton direction to be taken at each iteration (with an appropriate stepsize).
}

\noindent The solution of \eqref{eq:Newton_System} is delivered by the following computational procedure 
\begin{align}
	& \big(A(\Theta^{-1}+\rho I)^{-1}A^T+\delta I\big)\Delta \mathbf{y} = \boldsymbol{\xi}_p - A(\Theta^{-1}+\rho I)^{-1}(X^{-1}\boldsymbol{\xi}_{\mu,\sigma} + \boldsymbol{\xi}_d )  \label{eq:normal_system} \\
	&  (\Theta^{-1}+\rho I) \Delta \mathbf{x} = A^T \Delta \mathbf{y} + \boldsymbol{\xi}_d + X^{-1}\boldsymbol{\xi}_{\mu,\sigma} \label{eq:normal_sol_2_bis}  \\ 
	& X \Delta \mathbf{s} = (\boldsymbol{\xi}_{\mu,\sigma} - S \Delta \mathbf{x})  \label{eq:normal_sol_3_bis}.
\end{align}
where $\Theta:=XS^{-1}$.  Before continuing let us give basic definitions used in the remainder of this work.

\begin{definition}
	Normal Matrix:
	\begin{equation}
	\label{eqn:normal_equations_matrix}
		S_{\rho, \delta}:=A(\Theta^{-1}+\rho I)^{-1}A^T+\delta I.
	\end{equation}
Neighbourhood of the infeasible central path:
\begin{equation}\label{eqn:neighbourhood}
	\begin{split}
	\mathcal{N}_k(\bar {\gamma},\underline{\gamma},\gamma_p,\gamma_d):= & \{(\mathbf{x},\mathbf{y},\textbf{s}) \in  {\editSte \mathbb{R}^n_{>0} \times \mathbb{R}^m \times \mathbb{R}^n_{>0} } \;:\; \\
	& \bar{\gamma} \mathbf{x}^T\mathbf{s}/n  \geq x_is_i \geq \underline{\gamma} \mathbf{x}^T\mathbf{s}/n\; \hbox{ for } \; i=1,\dots,n; \\
	&  \mathbf{x}^T\mathbf{s} \geq \gamma_p \|A\mathbf{x}+\delta(\mathbf{y}-\mathbf{y}_k)-\mathbf{b}\|; \\
	&  \mathbf{x}^T\mathbf{s} \geq  \gamma_d\|\rho(\mathbf{x}-\mathbf{x}_k)-A^T\mathbf{y}-\mathbf{s} + \mathbf{c}\|\}, 
	\end{split}
\end{equation}
where $\bar{\gamma} > 1 >\underline{\gamma}> 0$ and $(\gamma_p,\gamma_d)>0$.
\end{definition}

{\editSte
The neighbourhood here considered is standard in the analysis of infeasible IPMs \cite{MR1422257}: it requires the iterates to be close enough to the central path (according to parameters $\bar\gamma$ and $\underline\gamma$), and the primal-dual constraint violations to be reduced at the same rate as the complementarity product $\mathbf x^T\mathbf s$. Within this neighbourhood, $\mathbf x^T\mathbf s\to0$ guarantees convergence to a primal-dual optimal solution.
}

Moreover, we consider an {\it inexact} solution of the linear system \eqref{eq:normal_system}:
\begin{assumption} \label{ass:residual}
	\begin{equation} \label{eq:inexact_assumption}
		S_{\rho, \delta} \Delta \mathbf{y} = \bar{\boldsymbol{\xi}}_p + \boldsymbol{\boldsymbol{\zeta}} \hbox{ where } \|\boldsymbol{\boldsymbol{\zeta}}\|\le C_\text{inexact}\, {\mathbf{x}}^T\mathbf{s},
	\end{equation}
where $C_\text{inexact}\in(0,1)$ and we defined 
\begin{equation*}
	\bar{\boldsymbol{\xi}}_p:= \boldsymbol{\xi}_p - A(\Theta^{-1}+\rho I)^{-1}(X^{-1}\boldsymbol{\xi}_{\mu,\sigma} + \boldsymbol{\xi}_d ).
\end{equation*}
\end{assumption} 
It is important to note that the above Assumption \ref{ass:residual} is a non-standard requirement in inexact Newton methods \cite{Kel:book,DemEisSte:inexact_newton}. Its particular form is motivated by the use of IPM and the needs of the complexity analysis in Section \ref{sec:convergence}. It is chosen in agreement 
with the definition of the infeasible neighbourhood \eqref{eqn:neighbourhood} of the central path of the sub-problem considered. %Therefore reducing the complementarity product  $\mathbf{x}^T \mathbf{s}$  guarantees convergence of the inexact IPM. 
Using \eqref{eq:normal_sol_2_bis} and \eqref{eq:inexact_assumption} in \eqref{eq:normal_sol_2}, we have 

\begin{equation*}
		 \|  A\Delta \mathbf{x} + \delta \Delta \mathbf{y} -  \boldsymbol{\xi}_p  \| =
%		& \|  A  (\Theta^{-1}+\rho I)^{-1}(A^T \Delta \mathbf{y} + \boldsymbol{\xi}_d + X^{-1}\boldsymbol{\xi}_{\mu,\sigma}) + \delta \Delta \mathbf{y} -  \boldsymbol{\xi}_p \| = \\
		  \| S_{\rho, \delta} \Delta \mathbf{y} - \bar{\boldsymbol{\xi}}_p \| = \|\boldsymbol{\boldsymbol{\zeta}}\|,
\end{equation*}
whereas equations \eqref{eq:Newton_details} and \eqref{eq:normal_sol_3} are satisfied exactly. Therefore the inexact Newton directions computed according to \eqref{eq:inexact_assumption} satisfy:

\begin{equation} \label{eq:inexact_Newton_System}
	\begin{bmatrix}
		\rho I & -A^T & -I \\
		A    & \delta I & 0 \\
		S      &  0       & X
	\end{bmatrix}\begin{bmatrix}
		\Delta \mathbf{x} \\
		\Delta \mathbf{y} \\
		\Delta \mathbf{s}
	\end{bmatrix} =\begin{bmatrix}
		\boldsymbol{\xi}_d \\
		\boldsymbol{\xi}_p \\
		\boldsymbol{\xi}_{\mu,\sigma}
	\end{bmatrix}+\begin{bmatrix}
	0 \\
	\boldsymbol{\boldsymbol{\zeta}} \\
	0
\end{bmatrix}.
\end{equation}

\noindent Define
\begin{equation} \label{eq:alpha_direc}
	\begin{bmatrix}
		\mathbf{x}^j_k(\alpha)\\
		\mathbf{y}^j_{k}(\alpha) \\
		\mathbf{s}^j_{k}(\alpha) 
	\end{bmatrix}:=\begin{bmatrix}
		\mathbf{x}_{k}^j\\
		\mathbf{y}_{k}^j \\
		\mathbf{s}_{k}^j 
	\end{bmatrix}+\begin{bmatrix}
		\alpha \Delta \mathbf{x}_k^j\\
		\alpha \Delta \mathbf{y}_k^j \\
		\alpha \Delta \mathbf{s}_k^j 
	\end{bmatrix},
\end{equation}
{\edit i.e.\ $\mathbf{x}^j_k(\alpha)$ is the point reached from $\mathbf{x}^j_k$ after a step of length $\alpha$ along the Newton direction. Notice that, after selecting the correct stepsize $\alpha^j_k$, we define $\mathbf{x}^{j+1}_k:=\mathbf{x}^j_k(\alpha^j_k)$.}

We report in Algorithm \ref{alg:IPM} a prototype IPM scheme for the solution of problem \eqref{eq:LP_r_problem}. {\edit The fundamental steps involved in the algorithm are: computing the Newton direction by solving \eqref{eq:inexact_Newton_System} with a level of inexactness that satisfies \eqref{eq:inexact_assumption}, see Line \ref{algline:Newton}; finding the largest stepsize that guarantees to remain inside the neighbourhood and to sufficiently reduce the complementarity products, see Line \ref{algline:step_lenght}; preparing the quantities to be used in the next iteration, see Lines \ref{algline:newton_step}-\ref{algline:new_complementarity}.}

We study the convergence of Algorithm~\ref{alg:IPM} in Section \ref{sec:convergence}. Concerning the notation, recall that the subscript $k$ is related to the iteration count of the outer Algorithm \ref{alg:PS-MF-IPM} (PPM) whereas the superscript $j$ is related to the iteration of the inner Algorithm \ref{alg:IPM} (IPM). {\edit To avoid over-complicating the notation, notice that in the following we use $\boldsymbol{\xi}_{p,k}^j$ and $\boldsymbol{\xi}_{d,k}^j$ instead of $(\boldsymbol{\xi}_p)_k^j$ and $(\boldsymbol{\xi}_d)_k^j$.}

%The method has a guaranteed polynomial convergence \cite[Chap. 6]{MR1422257} (cfr. also \cite{MR2899152,MR1242461,MR2500832,MR3082499}). 

\begin{algorithm}[hbt!]
	\caption{IPM: inner loop of PS-IPM}\label{alg:IPM}
	\KwIn{ 
		$0<\sigma,\bar{\sigma}<1$, barrier reduction parameters s.t. $\sigma < \bar{\sigma}$; \\
		$C_\text{inexact} \in (0,1)$ inexactness parameter s.t. $ \gamma_pC_\text{inexact} < {\sigma}$; \\
		%$\varepsilon_{p,k}>0,\varepsilon_{d,k}>0,\varepsilon_{c,k}>0$ optimality
		%tolerances;
		}
	\init{\\
		Iteration counter $j = 0$; \\
		Primal–dual point $(\mathbf{x}_k^0,\mathbf{y}_k^0,\mathbf{s}_k^0)\in \mathcal{N}_k(\bar {\gamma},\underline{\gamma},\gamma_p,\gamma_d)$\\
		Compute $\mu_k^0:={\mathbf{x}_k^0}^T\mathbf{s}_k^0/n$, $\boldsymbol{\xi}^0_{d,k}$, and $\boldsymbol{\xi}^0_{p,k}$.}
	%	\While{$\|\boldsymbol{\xi}^j_{p,k}\|>\varepsilon_{p,k}$ or $\|\boldsymbol{\xi}^j_{d,k}\|>\varepsilon_{d,k}$ or ${\mathbf{x}^j_k}^T\mathbf{s}^j_k>\varepsilon_{c,k}$}{
		\While{{\editSte \texttt{Stopping Criterion \eqref{eq:stopping_condition}}  False}}{
			Solve the KKT system \eqref{eq:inexact_Newton_System} using  $[\boldsymbol{\xi}^j_{d,k},\boldsymbol{\xi}^j_{p,k},\boldsymbol{\xi}^j_{\mu_k^j, \sigma}]^T$  with $\|\boldsymbol{\boldsymbol{\zeta}}_k^{j}\| \leq C_\text{inexact} (\mathbf{x}^j_k)^T\mathbf{s}_k$ to find $[\Delta \mathbf{x}_k^j,\; \Delta \mathbf{y}_k^j, \; \Delta \mathbf{s}^j_k ]^T$ \label{algline:Newton} \;
			{\editSte  
			Compute
			\begin{equation*}
				\alpha_p^{*,j} = \sup \{\alpha \in \mathbb{R} \;:\;  \mathbf{x}_k^j(\alpha) \geq 0 \}
			\end{equation*}
			\begin{equation*}
				\alpha_d^{*,j} = \sup \{\alpha \in \mathbb{R} \;:\;  \mathbf{s}_k^j(\alpha) \geq 0 \}
			\end{equation*}
			and define $\alpha^{*,j} := \{ \alpha_p^{*,j}, \alpha_d^{*,j}\} $\;
			\If{ $(\mathbf{x}_k(\alpha^{*,j}),\mathbf{y}_k(\alpha^{*,j}),\mathbf{s}_k(\alpha^{*,j}))$ is a solution of \eqref{eq:LP_r_problem} \label{eq:alg_stopping_condition} } 
			{Stop} 
			}

			Find $\alpha_k^j$ as the maximum $\alpha \in [0,1]$ s.t. {\editSte for all $\alpha \in [0,\alpha_k^j]$}
			\begin{equation}\label{eq:IPM_next_Step}
				\begin{split}
					& (\mathbf{x}_k^j(\alpha),\mathbf{y}^j_k(\alpha),\mathbf{s}^j_k(\alpha) ) \in \mathcal{N}_k(\bar {\gamma},\underline{\gamma},\gamma_p,\gamma_d)   \hbox{ and }  \\
					& \mathbf{x}^j_k(\alpha)^T\mathbf{s}_k^j(\alpha) \leq (1 -(1-\bar{\sigma})\alpha ){\mathbf{x}_k^j}^T\mathbf{s}_k^j \hbox{\; }
				\end{split}
			\end{equation} \label{algline:step_lenght} 
			
			Set $\begin{bmatrix}
				\mathbf{x}_{k}^{j+1}\\
				\mathbf{y}_{k}^{j+1} \\
				\mathbf{s}_{k}^{j+1} 
			\end{bmatrix}=\begin{bmatrix}
				\mathbf{x}_{k}^{j}\\
				\mathbf{y}_{k}^{j} \\
				\mathbf{s}_{k}^{j} 
			\end{bmatrix}+\begin{bmatrix}
				\alpha_k^j \Delta \mathbf{x}_k^j\\
				\alpha_k^j \Delta \mathbf{y}_k^j \\
				\alpha_k^j \Delta \mathbf{s}_k^j 
			\end{bmatrix}$ \label{algline:newton_step} \;
			Compute the infeasibilities $\boldsymbol{\xi}^{j+1}_{d,k}$, $\boldsymbol{\xi}^{j+1}_{p,k}$ and barrier parameter $\mu^{j+1}_k:={\mathbf{x}^{j+1}_k}^T\mathbf{s}^{j+1}_k/n$  \label{algline:new_complementarity}\;
			Update the iteration counter: $j := j + 1$.
		}
	\end{algorithm}

\section{Convergence and complexity} \label{sec:convergence}
In this section, we show that the particular inexact IPM in Algorithm \ref{alg:IPM}, used as inner solver  in Algorithm \ref{alg:PS-MF-IPM}, is convergent. Moreover, at the end of the present section, we show that such IPM converges to an $\varepsilon-$accurate solution in a polynomial number of iterations. Our implant of the proof is inspired by the works \cite{MR3010439,MR1608030,Cornelis_Vanrose,MR2899152,MR1242461, MR1785643} but consistently differs from the hypothesis and techniques used there.

The PPM iteration counter $k$ is fixed through this section and, for the sake of readability, is used only when writing the fixed PPM iteration $(\mathbf{x}_k,\mathbf{y}_k,\mathbf{s}_k)$ and not in the context of the IPM iterations $(\mathbf{x}^j_k,\mathbf{y}^j_k,\mathbf{s}^j_k)$. 

We start from analysing the progress made in a single Newton iteration. Using \eqref{eq:KKT_map}, \eqref{eq:Newton_System}, and \eqref{eq:Newton_details} we obtain
\begin{equation}\label{eq:dual_progress}
	\begin{split}
		&\rho(\mathbf{x}^{j}(\alpha)-\mathbf{x}_{ k})-A^T\mathbf{y}^{j}{(\alpha)}-\mathbf{s}^{j}{(\alpha)}  + \mathbf{c} \\
		& = (\rho(\mathbf{x}^{j}-\mathbf{x}_{ k})-A^T\mathbf{y}^{j}-\mathbf{s}^{j} + \mathbf{c})  +\alpha (\rho \Delta \mathbf{x}^{j} - A^T \Delta \mathbf{y}^{j} - \Delta \mathbf{s}^{j} )  \\
		&= (1-\alpha)(\rho(\mathbf{x}^{j}-\mathbf{x}_{ k})-A^T\mathbf{y}^{j}-\mathbf{s}^{j} + \mathbf{c}),
	\end{split} \;\;\;
\end{equation}
whereas, using \eqref{eq:KKT_map} {\edit and \eqref{eq:inexact_Newton_System}} we have
\begin{equation} \label{eq:primal_progress}
	\begin{split}
		& A\mathbf{x}^j(\alpha) +\delta(\mathbf{y}^j(\alpha)  -\mathbf{y_{k}})-\mathbf{b}  \\
		&  =(A\mathbf{x}^j +\delta(\mathbf{y}^j - \mathbf{y_{ k}})-\mathbf{b})	+ \alpha (A \Delta \mathbf{x}^{j}+\delta \Delta \mathbf{y}^{j})  \\
		&  = (1- \alpha)(A\mathbf{x}^j +\delta(\mathbf{y}^j - \mathbf{y_{ k}})-\mathbf{b}) + \alpha \boldsymbol{\boldsymbol{\zeta}}^{j}.
	\end{split}
\end{equation}

\noindent The last block equation in \eqref{eq:inexact_Newton_System} yields
\begin{equation} \label{eq:third_block_1}
		 (\mathbf{s}^j)^T\Delta \mathbf{x}^{j}+(\mathbf{x}^j)^T\Delta \mathbf{s}^{j}
		 =- (\mathbf{x}^j)^{T}\mathbf{s}^j + \sigma n \mu^j 
		 =  (\sigma-1)(\mathbf{x}^j)^{T}\mathbf{s}^j
\end{equation}
and
\begin{equation}\label{eq:third_block_2}
	  s_i \Delta x_i  + x_i \Delta s_i = \sigma \frac{\mathbf{x}^T\mathbf{s}}{n} - x_is_i\; \; .
\end{equation}

\noindent Finally, using  \eqref{eq:third_block_1}, we state the following identity
\begin{equation} \label{eq:scalar_product_with_step}
	(\mathbf{x}^j+\alpha \Delta \mathbf{x}^j)^T(\mathbf{s}^j+\alpha \Delta \mathbf{s}^j)= (\mathbf{x}^j)^T\mathbf{s}^j(1+\alpha(\sigma-1))+\alpha^2 (\Delta \mathbf{x}^j)^T\Delta \mathbf{s}^j.
\end{equation}

{\editSte
With the next Theorem \ref{th:IPM_well_definitness}  we prove that Algorithm \ref{alg:IPM} is well-defined: at each iteration, there exist a non-empty interval of values for the stepsize $\alpha$ such that the next iterate still lies in the neighbourhood $\mathcal N_k(\bar\gamma,\underline\gamma,\gamma_p,\gamma_d)$ and such that the complementarity product $(\mathbf x_k^j)^T\mathbf s_k^j$ is reduced by a sufficient amount, as required in \eqref{eq:IPM_next_Step}. 
}

\begin{theorem} \label{th:IPM_well_definitness}
	{\editSte Let us suppose that $(\mathbf{x}^j,\mathbf{y}^j,\mathbf{s}^j) \in  \mathcal{N}_k(\bar {\gamma},\underline{\gamma},\gamma_p,\gamma_d)$ s.t.  $(\mathbf{x}^j)^T\mathbf{s}^j>0$ is given. If the stopping conditions at Line \ref{eq:alg_stopping_condition} of Algorithm~\ref{alg:IPM} are not satisfied, then there exists $0<\hat{\alpha}^j {\edit < \alpha^{*,j}}$ such that conditions \eqref{eq:IPM_next_Step} are satisfied for all $\alpha \in [0,\hat{\alpha}^j]$.}
\end{theorem}
\begin{proof}
	In this proof we omit also the IPM iterate counter $j$, i.e. $(\mathbf{x}^j,\mathbf{y}^j,\mathbf{s}^j)\equiv (\mathbf{x},\mathbf{y},\mathbf{s})$. Let us define the following functions, for all $i=1,\dots,n$
	\begin{align*}
		f_i(\alpha):= & (x_i+\alpha \Delta x_i)(s_i+\alpha \Delta s_i)- \underline{\gamma}(\mathbf{x}+\alpha \Delta \mathbf{x})^T(\mathbf{s}+\alpha \Delta \mathbf{s})/n, \\
		\bar{f}_i(\alpha):= & \bar{\gamma}(\mathbf{x}+\alpha \Delta \mathbf{x})^T(\mathbf{s}+\alpha \Delta \mathbf{s})/n-(x_i+\alpha \Delta x_i)(s_i+\alpha \Delta s_i),\\ 
		h(\alpha):= &(1 -(1-\bar{\sigma})\alpha) \mathbf{x}^T\mathbf{s} - (\mathbf{x}+\alpha \Delta \mathbf{x})^T(\mathbf{s}+\alpha \Delta \mathbf{s}), \\
		g_d(\alpha):= & (\mathbf{x}+\alpha \Delta \mathbf{x})^T(\mathbf{s}+\alpha \Delta \mathbf{s})  \nonumber &
		\\ &- \gamma_d \| \rho(\mathbf{x}+\alpha \Delta \mathbf{x} - \mathbf{x}_k) - A^T(\mathbf{y}+\alpha \Delta \mathbf{y}) -(\mathbf{s}+\alpha \Delta \mathbf{s})+\mathbf{c}\|, \\
		g_p(\alpha):= & (\mathbf{x}+\alpha \Delta \mathbf{x})^T(\mathbf{s}+\alpha \Delta \mathbf{s})   \nonumber \\
		              &- \gamma_p \| A(\mathbf{x}+\alpha \Delta \mathbf{x})+ \delta (\mathbf{y}+\alpha \Delta \mathbf{y}- \mathbf{y}_k) -\mathbf{b}  \| .
	\end{align*}

\noindent Using \eqref{eq:dual_progress} in the expressions of $g_d(\alpha)$ we have
\begin{equation*}
	g_d(\alpha) =(\mathbf{x}+\alpha \Delta \mathbf{x})^T(\mathbf{s}+\alpha \Delta \mathbf{s}) -\gamma_d (1-\alpha) \|\rho(\mathbf{x}-\mathbf{x}_k)-A^T\mathbf{y}-\mathbf{s} + \mathbf{c}\|,
\end{equation*}
whereas using \eqref{eq:primal_progress} in the expressions of $g_p(\alpha)$ we have
\begin{equation*}
	g_p(\alpha) \geq(\mathbf{x}+\alpha \Delta \mathbf{x})^T(\mathbf{s}+\alpha \Delta \mathbf{s}) -   \gamma_p ((1-\alpha) \|A\mathbf{x} +\delta(\mathbf{y} - \mathbf{y}_k)-\mathbf{b}\| + \alpha \|\boldsymbol{\boldsymbol{\zeta}}\|).
\end{equation*}

{\editSte We start proving that} there exists $\hat{\alpha}^{j}>0$ such that
\begin{equation*}
	f_i(\alpha)\geq 0, \; \bar{f}_i(\alpha)\geq 0, \; h(\alpha)\geq 0, \; g_p(\alpha)\geq 0, \; g_d(\alpha)\geq 0
\end{equation*}
for all $i = 1, \dots, n$ and for all $\alpha \in [0, \hat{\alpha}^j]$. In the following we will use exensively the identity \eqref{eq:scalar_product_with_step}. We have
\begin{equation} \label{eq:ineq_1}
	\begin{split}
		 f_{i}(\alpha)= &\underbrace{(1-\alpha)(x_is_i-\underline{\gamma} \frac{\mathbf{x}^T\mathbf{s}}{n})}_{\geq 0}+\alpha^2(\Delta x_i \Delta s_i - \underline{\gamma}\frac{(\Delta \mathbf{x})^T \Delta \mathbf{s}}{n}) + \alpha \sigma (1-\underline{\gamma})\frac{\mathbf{x}^T\mathbf{s}}{n} \\
		& \geq \alpha^2(\Delta x_i \Delta s_i - \underline{\gamma}\frac{(\Delta \mathbf{x})^T \Delta \mathbf{s}}{n}) + \alpha \sigma (1-\underline{\gamma})\frac{\mathbf{x}^T\mathbf{s}}{n}.
	\end{split}
\end{equation}
Since $\mathbf{x}^T\mathbf{s}>0$, using a simple continuity argument, we can infer the existence of a small enough 
$\underline{f}_i>0$ s.t. $ f_{i}(\alpha)\geq 0$ for all $\alpha \in [0, \underline{f}_i]$. Reasoning analogously, we have
\begin{equation} \label{eq:ineq_2}
	\begin{split}
		& \bar{f}_{i}(\alpha)= \underbrace{(1-\alpha)(\bar{\gamma} \frac{\mathbf{x}^T\mathbf{s}}{n} -x_is_i)}_{\geq 0}+\alpha^2(\bar{\gamma}\frac{(\Delta \mathbf{x})^T \Delta \mathbf{s}}{n}-\Delta x_i \Delta s_i ) + \alpha \sigma (\bar{\gamma}-1)\frac{\mathbf{x}^T\mathbf{s}}{n} \\
		& \geq \alpha^2(\bar{\gamma}\frac{(\Delta \mathbf{x})^T \Delta \mathbf{s}}{n}-\Delta x_i \Delta s_i ) + \alpha \sigma (\bar{\gamma}-1)\frac{\mathbf{x}^T\mathbf{s}}{n},
	\end{split}
\end{equation}
and hence there exists a small enough 
$\bar{f}_i>0$ s.t. $ \bar{f}_{i}(\alpha)\geq 0$ for all $\alpha \in [0, \bar{f}_i]$.

Concerning $h(\alpha)$, we have
\begin{equation} \label{eq:ineq_3}
	h(\alpha)= \mathbf{x}^T\mathbf{s}(\bar{\sigma}-\sigma) \alpha- \alpha^2 (\Delta \mathbf{x})^T\Delta \mathbf{s},
\end{equation}
and, since $\mathbf{x}^T\mathbf{s}(\bar{\sigma}-\sigma)>0$, there exists $\hat{h}>0$ small enough s.t. $h(\alpha)\geq 0$ for all $\alpha \in [0, \hat{h}]$.

Concerning $g_d(\alpha)$, we have
\begin{equation} \label{eq:ineq_4}
	\begin{split}
		g_d(\alpha) = & \underbrace{(1-\alpha)(\mathbf{x}^T\mathbf{s}- \gamma_d \|(\rho(\mathbf{x}-\mathbf{x}_k)-A^T\mathbf{y}-\mathbf{s} + \mathbf{c})\|)}_{\geq 0} + \alpha \sigma \mathbf{x}^T\mathbf{s} +\alpha^2 (\Delta \mathbf{x})^T\Delta \mathbf{s} \\
		& \geq \alpha \sigma \mathbf{x}^T\mathbf{s} +\alpha^2 (\Delta \mathbf{x})^T\Delta \mathbf{s},
	\end{split}
\end{equation}
and hence there exists $\hat{g}_d>0$ small enough s.t. $g_d(\alpha)\geq 0$ for all $\alpha \in [0, \hat{g}_d]$.

Finally, concerning $g_p(\alpha)$, we have
\begin{equation} \label{eq:ineq_5}
	\begin{split}
		g_p(\alpha) \geq & \underbrace{(1-\alpha)(\mathbf{x}^T\mathbf{s}- \gamma_p \|(A\mathbf{x} +\delta(\mathbf{y} - \mathbf{y}_k)-\mathbf{b})\|)}_{\geq 0} + \alpha \sigma \mathbf{x}^T\mathbf{s} +\\
		&+\alpha^2 (\Delta \mathbf{x})^T\Delta \mathbf{s} - \alpha\gamma_p \|\boldsymbol{\boldsymbol{\zeta}}\| \\
		& \geq \alpha (\sigma - \gamma_p C_\text{inexact} )\mathbf{x}^T\mathbf{s} +\alpha^2 (\Delta \mathbf{x})^T\Delta \mathbf{s},
	\end{split}
\end{equation}
and hence there exists $\hat{g}_p>0$ small enough s.t. $g_p(\alpha)\geq 0$ for all $\alpha \in [0, \hat{g}_p]$. 
{\editSte
	Let us define
\begin{equation*}
	\hat{\alpha}^j = \min \{ \min_{i} \underline{f}_i, \;\; \min_{i} \bar{f}_i,  \;\; \hat{h}, \;\; \hat{g}_d, \;\; \hat{g}_p, \; 1\} >0.
\end{equation*}
To prove the thesis, it remains to show that $\alpha^{*,j}>\hat{\alpha}^j$, i.e.\ that
\begin{equation*}
(\mathbf{x}(\alpha),\mathbf{y}(\alpha),\mathbf{s}(\alpha) ) \in \mathbb{R}^n_{>0} \times \mathbb{R}^m \times \mathbb{R}^n_{>0}	\hbox{ for all } \alpha \in [0,\hat{\alpha}^j].
\end{equation*}
To this aim, let us suppose by contradiction that $\alpha^{*,j} \leq \hat{\alpha}^j$. By  definition of $\alpha^{*,j}$, there exists $\bar{\ell} \in \{1, \dots,n\}$ s.t. $(x_{\bar{\ell}}+\alpha^{*,j} \Delta x_{\bar{\ell}})(s_{\bar{\ell}}+\alpha^{*,j} \Delta s_{\bar{\ell}})=0$. We have hence
\begin{equation*}
	f_{\bar{\ell}}(\alpha^{*,j})= - \underline{\gamma}(\mathbf{x}\big(\alpha^{*,j})\big)^T\mathbf{s}(\alpha^{*,j})/n \geq 0 \Rightarrow (\mathbf{x}\big(\alpha^{*,j})\big)^T\mathbf{s}(\alpha^{*,j}) =0. 
\end{equation*}
From the above implication,  using $g_d(\alpha^{*,j})$ and $g_p(\alpha^{*,j})$, we obtain that
\begin{equation*}
	\begin{split}
		& A\mathbf{x}(\alpha^{*,j}) +\delta(\mathbf{y}(\alpha^{*,j}) -\mathbf{y}_k) - \mathbf{b} =0 \\
		& \rho(\mathbf{x}(\alpha^{*,j}) - \mathbf{x}_k)-A^T\mathbf{y}(\alpha^{*,j}) - \mathbf{s}(\alpha^{*,j}) +\mathbf{c}=0, 
	\end{split}
\end{equation*}
i.e.\ $(\mathbf{x}(\alpha^{*,j}),\mathbf{y}(\alpha^{*,j}),\mathbf{s}(\alpha^{*,j}))$ is a solution of problem \eqref{eq:LP_r_problem}. We have hence obtained a contradiction since we are supposing that Algorithm~\ref{alg:IPM} did not stop at Line \ref{eq:alg_stopping_condition}. 
}
\end{proof}

{\editSte
Before proving the convergence of Algorithm \ref{alg:IPM} we would like to emphasize that  the above proof complements and expands \cite[Remark 3.1]{MR1242461}.

The next two results establish that Algorithm \ref{alg:IPM}  converges to a solution of problem \eqref{eq:LP_r_problem}. This is done by establishing that the right-hand sides of the Newton systems are uniformly bounded and by showing that the complementarity product $(\mathbf{x}^j)^T\mathbf{s}^j$ cannot be bounded away from zero.
}

\begin{corollary}\label{rem:uniform_boundedness_rhs}
	The right-hand sides of the Newton systems are uniformly bounded. 
\end{corollary}
\begin{proof}
	As a consequence of Theorem \ref{th:IPM_well_definitness}, we can suppose the existence of a sequence of iterates  $\{(\mathbf{x}^j,\mathbf{y}^j,\mathbf{s}^j)\}_{j \in \mathbb{N}}$ produced by Algorithm \ref{alg:IPM} s.t. $$(\mathbf{x}^j,\mathbf{y}^j,\mathbf{s}^j) \in \mathcal{N}_k(\bar {\gamma},\underline{\gamma},\gamma_p,\gamma_d). $$
	Since by construction $(\mathbf{x}^j)^T\mathbf{s}^j \leq (\mathbf{x}^0)^T\mathbf{s}^0 $, we have from \eqref{eqn:neighbourhood}
	\begin{equation*}
		\begin{split}
			& \|A\mathbf{x}^j+\delta(\mathbf{y}^j-\mathbf{y}_k)-\mathbf{b}\| \leq (\mathbf{x}^0)^T\mathbf{s}^0/\gamma_p, \\
			& \|\rho(\mathbf{x}^j-\mathbf{x}_k)-A^T\mathbf{y}^j-\mathbf{s}^j + \mathbf{c}\| \leq (\mathbf{x}^0)^T\mathbf{s}^0/\gamma_d. 
		\end{split}
	\end{equation*}
	Moreover, we have
	\begin{equation*}
		\begin{split}
			& \|S^jX^j\mathbf{e}-\sigma \mu^j \mathbf{e}\| \leq \|S^jX^j\mathbf{e}\|+\sigma \mu^j \|\mathbf{e}\| \\
			& \leq \frac{\bar{\gamma}}{\sqrt{n}} (\mathbf{x}^j)^T\mathbf{s}^j +  \frac{\sigma}{\sqrt{n}} (\mathbf{x}^j)^T\mathbf{s}^j \leq \frac{\bar{\gamma}+\sigma}{\sqrt{n}} (\mathbf{x}^0)^T\mathbf{s}^0.
		\end{split} 
	\end{equation*} 
\end{proof}

\begin{theorem} \label{th:convergence}
 Algorithm \ref{alg:IPM} produces a sequence of iterates in $\mathcal{N}_k(\bar {\gamma},\underline{\gamma},\gamma_p,\gamma_d)$ s.t. $\lim  (\mathbf{x}^j)^T\mathbf{s}^j=0$, i.e. $(\mathbf{x}^j,\mathbf{y}^j,\mathbf{s}^j) $ converges to a solution of problem \eqref{eq:LP_r_problem}. 
\end{theorem}
\begin{proof}
 Let us argue by contradiction supposing that there exists $\varepsilon^*>0$ s.t. $(\mathbf{x}^j)^T\mathbf{s}^j> \varepsilon^*$ for all $j \in \mathbb{N}$.\\
 
 \textbf{Claim 1} There exists a constant $C_1$ dependent only on $n$ s.t. $$\|  [\Delta \mathbf{x}^j,\; \Delta \mathbf{y}^j, \; \Delta \mathbf{s}^j ]^T\|\leq C_1 \hbox{ for all } j \in \mathbb{N}.$$ 

 The proof of this fact follows observing that the Newton matrices in \eqref{eq:inexact_Newton_System} satisfy all the hypothesis of \cite[Th. 1]{MR3010439}, i.e. they have a uniformly bounded inverse, and that the right-hand sides are uniformly bounded, see Corollary~\ref{rem:uniform_boundedness_rhs}. As a consequence, there exists another constant $C_2$ s.t.
 \begin{equation} \label{eq:uniform_bound_deltas}
 	\begin{split}
 		& |\Delta x_i^j \Delta s_i^j - \underline{\gamma}\frac{(\Delta \mathbf{x}^j)^T \Delta \mathbf{s}^j}{n}| \leq C_2, \\
 		& |\bar{\gamma}\frac{(\Delta \mathbf{x}^j)^T \Delta \mathbf{s}^j}{n}-\Delta x_i^j \Delta s_i^j | \leq C_2, \\
 		& |(\Delta \mathbf{x}^j)^T\Delta \mathbf{s}^j| \leq C_2.
 	\end{split}
 \end{equation} 
 
 \textbf{Claim 2} There exists $\alpha^*>0$ s.t. $\alpha^j \geq \alpha^*$ for all $j \in \mathbb{N}$. \\
 
 Using \eqref{eq:uniform_bound_deltas} in equations \eqref{eq:ineq_1}, \eqref{eq:ineq_2}, \eqref{eq:ineq_3}, \eqref{eq:ineq_4}, \eqref{eq:ineq_5} we have
 \begin{align*}
 	f_i(\alpha) \geq &\; \alpha^2(\Delta x_i^j \Delta s_i^j - \underline{\gamma}\frac{(\Delta \mathbf{x}^j)^T \Delta \mathbf{s}^j}{n}) + \alpha \sigma (1-\underline{\gamma})\frac{(\mathbf{x}^j)^T\mathbf{s}^j}{n} \\
 	              \geq &\; - C_2 \alpha^2 +\alpha \sigma (1-\underline{\gamma}) \varepsilon^*/n,\\ 
 	\bar{f}_i(\alpha) \geq &\; (\bar{\gamma}\frac{(\Delta \mathbf{x}^j)^T \Delta \mathbf{s}^j}{n}-\Delta x_i^j \Delta s_i^j ) + \alpha \sigma (\bar{\gamma}-1)\frac{(\mathbf{x}^j)^T\mathbf{s}^j}{n} \\
 	              \geq &\; - C_2 \alpha^2 +\alpha \sigma (\bar{\gamma}-1) \varepsilon^*/n,\\ 
 	 h(\alpha)= &\;(\mathbf{x}^j)^T\mathbf{s}^j(\bar{\sigma}-\sigma) \alpha- \alpha^2 (\Delta \mathbf{x}^j)^T\Delta \mathbf{s}^j \\
 	          \geq  &\; \varepsilon^*(\bar{\sigma}-\sigma) \alpha - C_2 \alpha^2,\\
 	 g_d(\alpha) \geq  &\;  \alpha \sigma (\mathbf{x}^j)^T\mathbf{s}^j +\alpha^2 (\Delta \mathbf{x}^j)^T\Delta \mathbf{s}^j \geq \varepsilon^*\sigma \alpha - C_2 \alpha^2,  \\    
 	 g_p(\alpha) \geq  &\;   \alpha (\sigma - q \gamma_p ) (\mathbf{x}^j)^T\mathbf{s}^j +\alpha^2 (\Delta \mathbf{x}^j)^T\Delta \mathbf{s}^j \geq \varepsilon^*\alpha (\sigma - \gamma_p C_\text{inexact} ) - C_2 \alpha^2. \\       
 \end{align*}
 
 Hence $\alpha^j\ge\alpha^*$, where 
 \begin{equation}\label{eqn:alpha_lower_bound}
\alpha^*:=\min\bigg\{{\edit 1}, \frac{\sigma(1-\underline{\gamma})\varepsilon^* /n}{C_2},\frac{\sigma(\bar{\gamma}-1)\varepsilon^* /n}{C_2}, \frac{(\bar{\sigma}-\sigma)\varepsilon^*}{C_2},\frac{\sigma \varepsilon^*}{C_2}, \frac{(\sigma- \gamma_p C_\text{inexact}) \varepsilon^*}{C_2}\bigg\}.
 \end{equation}
 
The convergence claim follows observing that the inequality

\begin{equation*}
	\varepsilon^* \leq (\mathbf{x}^j)^T\mathbf{s}^j \leq (1 -(1-\bar{\sigma})\alpha^* )^j(\mathbf{x}^0)^T\mathbf{s}^0
\end{equation*}
leads to a contradiction for $j \to \infty$.
\end{proof}

\subsection{Polynomial Complexity}
{\editSte In this section we show that the number of iterations needed to reduce $\mu^j$ below a certain tolerance $\varepsilon$ grows polynomially with the size of the problem.} {\editSte Moreover, it is important to note that, from the definition of the central path $\mathcal{N}_k(\bar {\gamma},\underline{\gamma},\gamma_p,\gamma_d)$, the same number of iterations is sufficient to reduce also the primal and dual infeasibility below the tolerance $\varepsilon$.}

In the following we will omit the index $j$ when this does not lead to ambiguities. It is important to note that the linear system in \eqref{eq:inexact_Newton_System} can be written in an alternative form as follows

\begin{equation*} 
	\underbrace{\begin{bmatrix}
		\rho I & A^T & -I \\
		A    & -\delta I & 0 \\
		S      &  0       & X
	\end{bmatrix}}_{=:J}\begin{bmatrix}
		\Delta \mathbf{x} \\
		-\Delta \mathbf{y} \\
		\Delta \mathbf{s}
	\end{bmatrix} =\begin{bmatrix}
		\boldsymbol{\xi}_d \\
		\boldsymbol{\xi}_p + \boldsymbol{\boldsymbol{\zeta}}\\
		\boldsymbol{\xi}_{\mu,\sigma}
	\end{bmatrix}.
\end{equation*}
Using \cite[Remark 1]{MR3010439}, we have that

\begin{equation} \label{eq:Newton_System_Symmetric_inverse}
	{J}^{-1}=\begin{bmatrix}
		H^{-1}                  & \frac{1}{\delta}H^{-1}A^T                        & H^{-1}X^{-1} \\
		\frac{1}{\delta}AH^{-1} & \frac{1}{\delta^2}AH^{-1}A^T- \frac{1}{\delta} I & \frac{1}{\delta}AH^{-1}X^{-1} \\
		- \Theta^{-1} H ^{-1}      & -\frac{1}{\delta}\Theta^{-1}H^{-1}A^T                 & (I-\Theta^{-1}H^{-1})X^{-1}
	\end{bmatrix},
\end{equation}
where $H:= \rho I + \Theta^{-1} + \frac{1}{\delta}A^TA$. 

{\edit To prove polynomial complexity, we start by bounding the terms that appear in the expression \eqref{eq:Newton_System_Symmetric_inverse}. The next two technical results are useful in this sense.}

\begin{lemma}
We have that	
	\begin{equation*}
		\| H^{-1}\| \in O(1).
	\end{equation*}
\end{lemma}

\begin{proof}
	Using the Sherman-Morrison-Woodbury formula, we get
	\begin{equation} \label{eq:H_inverse}
	 H^{-1}= (\rho I + \Theta ^{-1})^{-1}-\frac{1}{\delta}(\rho I + \Theta ^{-1})^{-1}A^T\big(I + \frac{1}{\delta}A(\rho I + \Theta ^{-1})^{-1}A^T\big)^{-1}A(\rho I + \Theta ^{-1})^{-1}.
	\end{equation}
We observe
\begin{equation} \label{eq:unifrom_boundedness_regularized_diag}
	(\rho I + \Theta ^{-1})^{-1}_{ii}=\frac{\Theta_{ii}}{\rho \Theta_{ii} + 1}= \frac{\rho \Theta_{ii}}{\rho \Theta_{ii} + 1} \frac{1}{\rho}< \frac{1}{\rho}
\end{equation}
and hence
\begin{equation*}
	\begin{split}
		 \|H^{-1}\| \;\leq\; & \| (\rho I + \Theta ^{-1})^{-1} \|\cdot\\
		 &\cdot\Big(1 + \frac{1}{\delta}\|A^T\|\|A\| \|\big(I + \frac{1}{\delta}A(\rho I + \Theta ^{-1})^{-1}A^T)^{-1}\| \| (\rho I + \Theta ^{-1}\big)^{-1} \|\Big) \\
		\leq\; & \frac{1}{\rho}(1+ \frac{1}{\delta \rho}\|A^T\|\|A\|),
	\end{split}
\end{equation*}
where we used that $\|\big(I + \frac{1}{\delta}A(\rho I + \Theta ^{-1})^{-1}A^T\big)^{-1}\|  \leq 1$.
\end{proof}

\begin{corollary}
	We have that
	\begin{equation*}
		\|\Theta^{-1} H^{-1}\| \in O(1).
	\end{equation*}
	
\end{corollary}
\begin{proof}
	Using \eqref{eq:unifrom_boundedness_regularized_diag}, we observe that 
	$$\big(\Theta^{-1} (\rho I + \Theta ^{-1})^{-1}\big)_{ii}= \frac{1}{\rho \Theta_{ii}+1} $$
	and hence, using \eqref{eq:H_inverse}, we have
	\begin{equation*}
		\begin{split}
				\|\Theta^{-1} H^{-1}\| \leq \;& \|\Theta^{-1} (\rho I + \Theta ^{-1})^{-1}\| \;\cdot\\ 
				  &\Big(1+ \frac{1}{\delta}\|A^T\|\|A\| \|\big(1 + \frac{1}{\delta}A(\rho I + \Theta ^{-1})^{-1}A^T)^{-1}\| \| (\rho I + \Theta ^{-1}\big)^{-1} \|\Big) \\
				\leq \;& 1+ \frac{1}{\delta \rho}\|A^T\|\|A\|,  \\
		\end{split}
	\end{equation*}
where we used \eqref{eq:unifrom_boundedness_regularized_diag}.	
\end{proof}

{\edit
Therefore, we know that
\[ \|H^{-1}\| \leq C_3 \hbox{ and } \|\Theta^{-1} H^{-1}\| \leq C_4,
\]
for some positive constants $C_3$ and $C_4$. Let us define $C_5 := \max \{C_3, C_4\}$.
}

{\edit Now that we have bounded the terms in \eqref{eq:Newton_System_Symmetric_inverse}, we look for an upper bound on the norms of the Newton directions that depend polynomially on the size of the problem $n$. This is a crucial step to find a polynomial lower bound on the minimum stepsize \eqref{eqn:alpha_lower_bound} which leads to the polynomial complexity result mentioned at the beginning of this section.}

\begin{theorem}\label{thm:poly_Newton}
	There exists a 
	%polynomial {\editSte with positive coefficients} $p_1(n)$ of degree at most one 
{\edit positive constant $C_6$}	
	s.t. $$\|  [\Delta \mathbf{x}^j,\; \Delta \mathbf{y}^j, \; \Delta \mathbf{s}^j ]^T\|\leq C_6 n\sqrt{\mu^j} \hbox{ for all } j \in \mathbb{N}.$$ 
\end{theorem}

\begin{proof}
	Using \eqref{eq:Newton_System_Symmetric_inverse}, we have
	\begin{equation*} \label{eq:explicit_deltas}
		\begin{split}
			&\Delta \mathbf{x}^j = {(H^j)^{-1}\boldsymbol{\xi}^j_d +\frac{1}{\delta}(H^j)^{-1}A^T(\boldsymbol{\xi}^j_p+\boldsymbol{\boldsymbol{\zeta}}^j)}+{(H^j)^{-1} (X^j)^{-1} \boldsymbol{\xi}^j_{\sigma, \mu^j}}, \\
			&-\Delta \mathbf{y}^j =  {\frac{1}{\delta}A(H^j)^{-1}\boldsymbol{\xi}^j_d +\Big(\frac{1}{\delta^2}A(H^j)^{-1}A^T-\frac{1}{\delta}I\Big)(\boldsymbol{\xi}^j_p+\boldsymbol{\boldsymbol{\zeta}}^j)} +  { \frac{1}{\delta}A(H^j)^{-1} (X^j)^{-1} \boldsymbol{\xi}^j_{\sigma, \mu^j}},\\
%			& \Delta \mathbf{s}^j :=  \underbrace{-(\Theta^{j})^{-1}(H^j)^{-1}\boldsymbol{\xi}^j_d- \frac{1}{\delta}(\Theta^{j})^{-1}(H^j)^{-1}(\boldsymbol{\xi}^j_p+\boldsymbol{\boldsymbol{\zeta}}^j)}_{=:\mathbf{a}_{\mathbf{s}}^j}+\underbrace{(I-(\Theta^{j})^{-1}(H^j)^{-1})(X^j)^{-1}\boldsymbol{\xi}^j_{\sigma, \mu^j}}_{=:\mathbf{b}_{\mathbf{s}}^j}
			\end{split}
	\end{equation*}
%Using \eqref{eq:explicit_deltas}, we have that	
%\begin{equation*}
%	\begin{split}
%		0 \leftarrow \|\Delta \mathbf{x}^j \| & \geq | \| \mathbf{a}_{\mathbf{x}}^j\| - \| \mathbf{b}_{\mathbf{x}}^j\||  \\
%		0 \leftarrow \|\Delta \mathbf{y}^j \| & \geq | \| \mathbf{a}_{\mathbf{y}}^j\| - \| \mathbf{b}_{\mathbf{y}}^j\|| \\
%		0 \leftarrow \|\Delta \mathbf{s}^j \| & \geq | \| \mathbf{a}_{\mathbf{s}}^j\| - \| \mathbf{b}_{\mathbf{s}}^j\||
%	\end{split}
%\end{equation*}
and hence
\begin{equation}\label{eq:Delta_x_bound}
\small
	\begin{split}
		\|\Delta \mathbf{x}^j\| \leq\; & \|(H^j)^{-1}\|\|\boldsymbol{\xi}^j_d\| +\frac{1}{\delta}\|(H^j)^{-1}\|\|A^T\|\|\boldsymbol{\xi}^j_p+\boldsymbol{\zeta}^j\|\\
		 +& \|(H^j)^{-1/2} \|\|(H^j)^{-1/2}(X^j)^{-1/2}(S^j)^{1/2} \| \|(X^j)^{1/2}(S^j)^{1/2}\mathbf{e} - \sigma \mu^j  (X^j)^{-1/2}(S^j)^{-1/2}\mathbf{e} \| \\
		 \leq\; & \frac{C_5 n \mu^j}{\gamma_d} + \frac{C_5 \|A^T\| (1+\gamma_p C_\text{inexact}) n\mu^j}{\delta \gamma_p}+ C_5 \Big(\sqrt{\bar{\gamma}n }+\sigma\frac{\sqrt{n}}{\sqrt{\underline{\gamma}}}\Big) \sqrt{\mu^j} \\
		 \leq & {\edit \bigg(\frac{C_5 \sqrt{\mu^0}}{\gamma_d} + \frac{C_5 \|A^T\| (1+\gamma_p C_\text{inexact})\sqrt{\mu^0}}{\delta \gamma_p}+ C_5 \Big(\sqrt{\bar{\gamma} }+\sigma\frac{1}{\sqrt{\underline{\gamma}}}\Big) \bigg) n\sqrt{\mu^j}   } \\
		 \leq\; & {\edit C_{\Delta x}n \sqrt{\mu^j} ,}
	\end{split}
\end{equation}
where we used
\begin{equation*}
	\begin{split}
		&\|(X^j)^{1/2}(S^j)^{1/2}\mathbf{e} - \sigma \mu^j  (X^j)^{-1/2}(S^j)^{-1/2}\mathbf{e} \| \\ 
		\leq\; & \|(X^j)^{1/2}(S^j)^{1/2}\mathbf{e}\|+ \|\sigma \mu^j  (X^j)^{-1/2}(S^j)^{-1/2}\mathbf{e} \|, \\
	\end{split}
\end{equation*}
{\edit $\mu^j=\sqrt{\mu^j}\,\sqrt{\mu^j}\le\sqrt{\mu^0}\,\sqrt{\mu^j}$ and $\sqrt{n}\le n$,}
and where $C_{\Delta x}$ is a positive constant.

Analogously
\begin{equation}\label{eq:Delta_y_bound}
	\begin{split}
		\|\Delta \mathbf{y}^j\| \leq\; & \frac{1}{\delta} \|A\|\|(H^j)^{-1}\|\|\boldsymbol{\xi}^j_d\| + \Big(\frac{1}{\delta^2}  \|(H^j)^{-1}\|\|A\| \|A^T\|+ \frac{1}{\delta} \Big) \| \boldsymbol{\xi}^j_p+\boldsymbol{\zeta}^j\|\\		
		+& \frac{1}{\delta}\|A\|\|(H^j)^{-1/2} \|\|(H^j)^{-1/2}(X^j)^{-1/2}(S^j)^{1/2} \|\cdot\\
		&\cdot\|(X^j)^{1/2}(S^j)^{1/2}\mathbf{e} - \sigma \mu^j  (X^j)^{-1/2}(S^j)^{-1/2}\mathbf{e} \| \\	
		\leq\; & \frac{\|A\|C_5 n \mu^j}{\delta\gamma_d}+ \frac{(C_5 \|A^T\|\|A\|+\delta )(1+\gamma_p C_\text{inexact})n \mu^j}{\delta^2 \gamma_p}+\\
		&+ \frac{C_5\|A\|}{\delta} \Big(\sqrt{\bar{\gamma}n }+\sigma\frac{\sqrt{n}}{\sqrt{\underline{\gamma}}}\Big) \sqrt{\mu^j} \\
		\leq\; & { \edit \bigg(\frac{\|A\|C_5 \sqrt{\mu^0}}{\delta\gamma_d}+ \frac{(C_5 \|A^T\|\|A\|+\delta )(1+\gamma_p C_\text{inexact}) \sqrt{\mu^0}}{\delta^2 \gamma_p}+}\\
		&{ \edit + \frac{C_5\|A\|}{\delta} \Big(\sqrt{\bar{\gamma} }+\sigma\frac{1}{\sqrt{\underline{\gamma}}}\Big) \bigg)n \sqrt{\mu^j}  }\\
		\leq\; & {\edit C_{\Delta y}n \sqrt{\mu^j} ,}
	\end{split}
\end{equation}	
where $C_{\Delta y}$ is a positive constant.
Finally, using the fact that $\Delta \mathbf{s}^j= \rho \Delta \mathbf{x}^j-A^T \Delta \mathbf{y}^j -\boldsymbol{\xi}_d^j $ and using \eqref{eq:Delta_x_bound}, \eqref{eq:Delta_y_bound}, and the definition of $\mathcal{N}_k(\bar {\gamma},\underline{\gamma},\gamma_p,\gamma_d)$, we have

\begin{equation*}\label{eq:Delta_s_bound}
	\begin{split}
		\| \Delta \mathbf{s}^j\|= & \rho \|\Delta \mathbf{x}^j\|+\|A^T\| \|\Delta \mathbf{y}^j\| +\|\boldsymbol{\xi}_d^j\| \\
		    \leq\; & {\edit \rho C_{\Delta x}n \sqrt{\mu^j} +\|A^T\| C_{\Delta y}n \sqrt{\mu^j}+\frac{n \mu^j}{\gamma_d}} \\
		    \leq\; & { \edit \Big(\rho C_{\Delta x} +\|A^T\| C_{\Delta y} +\frac{ \sqrt{\mu^0}}{\gamma_d}\Big)n\sqrt{\mu^j}}  \\
		    \leq\; & {\edit C_{\Delta s}n\sqrt{\mu^j},}
	\end{split}
\end{equation*}
where $C_{\Delta s}$ is a positive constant.

{\edit The thesis follows setting $C_6 = \max(C_{\Delta x},C_{\Delta y},C_{\Delta s})$.}
\end{proof}

{\edit The next straightforward technical result specializes the polynomial bound of Theorem \ref{thm:poly_Newton} to the terms that appear in equations \eqref{eq:ineq_1}-\eqref{eq:ineq_5}.}

\begin{corollary}
There exists a {\edit positive constant $C_7$} such that, for all $i$
\begin{equation} \label{eq:uniform_bound_deltas_polynomial}
	\begin{split}
		& \Big|\Delta x_i \Delta s_i - \underline{\gamma}\frac{(\Delta \mathbf{x})^T \Delta \mathbf{s}}{n}\Big| \leq C_7 n^2 \mu, \\
		& \Big|\bar{\gamma}\frac{(\Delta \mathbf{x})^T \Delta \mathbf{s}}{n}-\Delta x_i \Delta s_i \Big| \leq C_7 n^2 \mu, \\
		& |(\Delta \mathbf{x})^T\Delta \mathbf{s}| \leq C_7 n^2 \mu.
	\end{split}
\end{equation} 
\end{corollary}

{\edit We can now apply the previous results and obtain a bound similar to \eqref{eqn:alpha_lower_bound}, but that depends polynomially on the size of the problem $n$. This is the last fundamental step before the polynomial complexity result can be stated.}

\begin{theorem}
{\edit There exists a constant $\tilde \alpha$ s.t. $\alpha^j \geq \tilde\alpha$ for all $j \in \mathbb{N}$  and	
\begin{equation*}
	\tilde\alpha \geq  C_8 n^{-2},
\end{equation*}

where $C_8$ is a positive constant.}
\end{theorem}
\begin{proof}
	 Using \eqref{eq:uniform_bound_deltas_polynomial} in equations \eqref{eq:ineq_1}, \eqref{eq:ineq_2}, \eqref{eq:ineq_3}, \eqref{eq:ineq_4}, \eqref{eq:ineq_5} we have:
	\begin{align*}
		f_i(\alpha) \geq &\; \alpha^2\Big(\Delta x_i \Delta s_i - \underline{\gamma}\frac{(\Delta \mathbf{x})^T \Delta \mathbf{s}}{n}\Big) + \alpha \sigma (1-\underline{\gamma})\frac{\mathbf{x}^T\mathbf{s}}{n} \\
		\geq &\; - C_7n^2\mu \alpha^2 +\alpha \sigma (1-\underline{\gamma}) \mu,\\ 
		\bar{f}_i(\alpha) \geq &\; \alpha^2\Big(\bar{\gamma}\frac{(\Delta \mathbf{x})^T \Delta \mathbf{s}}{n}-\Delta x_i \Delta s_i \Big) + \alpha \sigma (\bar{\gamma}-1)\frac{\mathbf{x}^T\mathbf{s}}{n} \\
		\geq &\; - C_7 n^2 \mu \alpha^2 +\alpha \sigma (\bar{\gamma}-1) \mu,\\ 
		h(\alpha)= &\;\mathbf{x}^T\mathbf{s}(\bar{\sigma}-\sigma) \alpha- \alpha^2 (\Delta \mathbf{x})^T\Delta \mathbf{s}\geq n \mu (\bar{\sigma}-\sigma) \alpha - C_7 n^2 \mu  \alpha^2,\\
		g_d(\alpha) \geq  &\;  \alpha \sigma \mathbf{x}^T\mathbf{s} +\alpha^2 (\Delta \mathbf{x})^T\Delta \mathbf{s} \geq n \mu \sigma \alpha - C_7n^2 \mu \alpha^2,  \\    
		g_p(\alpha) \geq  &\;   \alpha (\sigma - \gamma_p C_\text{inexact}) \mathbf{x}^T\mathbf{s} +\alpha^2 (\Delta \mathbf{x})^T\Delta \mathbf{s} \geq n \mu \alpha (\sigma - \gamma_p C_\text{inexact} ) - C_7n^2 \mu \alpha^2. \\       
	\end{align*}
Hence, {\editSte defining}
	\begin{equation*}
		{\editSte \tilde\alpha :=}  \min\bigg\{{\edit 1}, \frac{\sigma(1-\underline{\gamma})}{C_7n^2},\;\frac{\sigma(\bar{\gamma}-1)}{C_7n^2},\; \frac{(\bar{\sigma}-\sigma)}{C_7n},\;\frac{\sigma}{C_7n},\; \frac{(\sigma- \gamma_p C_\text{inexact})}{C_7n}\bigg\},
	\end{equation*}
	{\editSte the thesis follows observing that, by definition, $\alpha^j \geq \tilde\alpha$.}
	
\end{proof}

{\edit Finally, we are ready to show that the number of iterations required to reduce $\mu$ below a certain tolerance $\varepsilon$ is proportional to $n^2$.}

\begin{theorem}
	Algorithm \ref{alg:IPM} has polynomial complexity, i.e. given $\varepsilon>0$ there exists $K \in O(n^2\ln (\frac{1}{\varepsilon}) )$ s.t. $\mu^ j \leq \varepsilon$ for all $j \geq K$.
\end{theorem}
\begin{proof}
	Thesis follows observing that
	\begin{equation*}
		(\mathbf{x}^j)^T\mathbf{s}^j \leq \big(1 -(1-\bar{\sigma})\tilde\alpha \big)^j(\mathbf{x}^0)^T\mathbf{s}^0 \leq \Big(1 -(1-\bar{\sigma}) \frac{C_8}{n^2} \Big)^j(\mathbf{x}^0)^T\mathbf{s}^0.
	\end{equation*}
\end{proof}

\section{Properties of the regularized normal equations system} \label{sec;properties_regularized_normal}
In this section we show some properties of matrix $S_{\rho,\delta}$ that are useful for the analysis performed in the next section.

{\edit
For the original graph $G(V,E)$, we define the {\it adjacency matrix} $\mathcal A\in\mathbb R^{|V|\times|V|}$ such that $\mathcal A_{ij}=1$ if there exists an edge between nodes $i$ and $j$ and $\mathcal A_{ij}=0$ otherwise. Notice that for an undirected graph $\mathcal A$ is symmetric; we assume that there are no self-loops, so that the diagonal of $\mathcal A$ is made of zeros. Let us define the {\it degree matrix} of a graph $\mathcal D\in\mathbb R^{|V|\times|V|}$ such that $\mathcal D$ is diagonal and $\mathcal D_{jj}$ is the degree of node $j$. Notice that $\mathcal D_{jj} = (\mathcal A\mathbf e)_j$.

Let us define also the {\it Laplacian matrix} of a graph as $\mathcal L\in\mathbb R^{|V|\times|V|}$ such that $\mathcal L=\mathcal D-\mathcal A$. An important relation between the Laplacian $\mathcal L$ and the node-arc incidence matrix $A$ is that $\mathcal L = AA^T$.

Given a diagonal matrix $\Theta$ and a parameter $\rho$, we define the re-weighted graph $G_\Theta$ as the graph with the same connectivity of $G$, in which the weight of every edge $j$ is scaled by a factor $\sqrt{\frac{\Theta_{jj}}{1+\rho\Theta_{jj}}}$. The adjacency matrix of the new graph $\mathcal A_\Theta$ has the same sparsity pattern of $\mathcal A$, but takes into account the new weight of the edges. The same happens for the degree matrix $\mathcal D_\Theta$. The incidence matrix therefore becomes $A_\Theta = A(\Theta^{-1}+\rho I)^{-1/2}$.

The new Laplacian matrix thus reads $\mathcal L_\Theta = \mathcal D_\Theta-\mathcal A_\Theta$ and can be written as 
\[
\mathcal L_\Theta = A_\Theta A_\Theta^T = A(\Theta^{-1}+\rho I)^{-1}A^T.
\]

We have just shown that the normal equations matrix can be interpreted as the Laplacian matrix of the original graph, where the edges are weighted according to the diagonal entries of matrix $\Theta$. This result is important because solving linear systems that involve Laplacian matrices is much easier than solving general linear systems. We summarize this result in the following Lemma.}
\begin{lemma} \label{lem:Laplacian_Graph}
	The matrix $A(\Theta^{-1}+\rho I)^{-1}A^T$ is the Laplacian of a weighted undirected graph and hence, for every $\Theta \in \mathbb{R}^{n\times n}_+$
	\begin{equation} \label{eq:Laplacian_IPM}
		A(\Theta^{-1}+\rho I)^{-1}A^T=\mathcal D_\Theta-\mathcal A_{\Theta},
	\end{equation}
	where $\mathcal D_\Theta$ and $\mathcal A_\Theta$ are the degree and adjacency matrices of the weighted graph.
\end{lemma}
%\begin{proof}
%This simply follows from the fact that $AA^T$ is the Laplacian of the original graph and $\tilde \Theta := (\Theta^{-1} + \rho I)^{-1}$   is a re-weighting of the edges.
%\end{proof}

{\edit The next result shows that the normal matrix is strictly diagonally dominant, due to the presence of dual regularization. This property is significant because it assures that} the incomplete Cholesky factorization of the normal equations matrix $S_{\rho,\delta}$ can always be computed without the algorithm breaking down (in exact arithmetic), see e.g.\ \cite{ichol_diag_dom}.

\begin{lemma}\label{lemma:diagdom}
	If $\delta>0$ the matrix $S_{\rho, \delta}$ is strictly diagonally dominant.
\end{lemma}

\begin{proof}
	From Lemma \ref{lem:Laplacian_Graph}, we have that $A(\Theta^{-1}+\rho I)^{-1}A^T=\mathcal D_{\Theta}-\mathcal A_{\Theta}$ and hence
	\begin{equation*}
		\sum_{j \neq i} |(S_{\rho, \delta})_{ij}|= \sum_{j \neq i} | -(\mathcal A_{\Theta})_{ij} |=\sum_{j \neq i} (\mathcal A_{\Theta})_{ij} < ({\mathcal D_{\Theta}})_{ii}+\delta = (S_{\rho, \delta})_{ii}.
	\end{equation*}
\end{proof}

{\edit The next two technical results are related to the distribution of eigenvalues of the normal matrix. They are used in the next section to show that the inexactness introduced when sparsifying the normal matrix remains bounded.}

\begin{lemma} \label{lem:max_eig_lap}
	$\lambda_{max}(AA^T) \leq 2 \max_{v \in V} deg(v)$.
\end{lemma}
\begin{proof}
	The proof is straightforward using Gershgorin's circle Theorems.
\end{proof}

\begin{lemma} \label{lem:eigs_schur}
	The eigenvalues $\lambda$ of matrix $S_{\rho,\delta}$ satisfy
	\[\delta \leq \lambda < \delta + \frac{2}{\rho}\max_{v \in V} deg(v).	\]
	
%	\begin{equation*}
%		\delta \leq \lambda( S_{\rho, \delta}) \leq 2 \max_{i} (\Theta^{-1}+\rho I)^{-1}_{ii} \max_{v \in V} deg(v) < \frac{2}{\rho}\max_{v \in V} deg(v).
%	\end{equation*} 
\end{lemma}
\begin{proof}
	Using the Rayleigh quotient, for some vector $\mathbf u$ and $\mathbf v=A^T\mathbf u$, the eigenvalues can be written as
	\[\lambda = 	\frac{\mathbf v^T(\Theta^{-1}+\rho I)^{-1}\mathbf v}{\mathbf v^T\mathbf v} \frac{\mathbf u^T AA^T\mathbf u}{\mathbf u^T\mathbf u} + \delta
		\]
	The lower bound $\lambda\ge\delta$ is trivial; the upper bound follows from Lemma \ref{lem:max_eig_lap} and \eqref{eq:unifrom_boundedness_regularized_diag}.
%	For the first inequality, use Rayleigh quotients. For the second inequality use Rayleigh quotients and Lemma \ref{lem:max_eig_lap}. For the last inequality use \eqref{eq:unifrom_boundedness_regularized_diag}. See also  the proof of Lemma \ref{lem:conditioning_Schir_Complements} for similar related techniques.
\end{proof}

\section{Sparsification of the reduced matrix} \label{sec:Sparsification}

{\edit
We now propose a technique to reduce the number of nonzeros in the normal equations $S_{\rho,\delta}$, based on the weights of the edges in the re-weighted graph (according to Lemma~\ref{lem:Laplacian_Graph}). We then show that this sparsification strategy is sound and produces a polynomially convergent interior point algorithm.
}

In this section we omit the IPM iteration counter $j$ and we  consider all the IPM-related quantities as a function of $\mu \to 0$. As IPMs progress towards optimality, we expect the following partition of the diagonal matrix $\Theta$ contributed by the barrier term:
\begin{equation*}
	\begin{split}
		& \mathcal{B}:=\{ i=1,\dots,n \hbox{ s.t. } x_i \to x_i^*>0, \;  s_i \to s_i^*=0  \}\\
		& \mathcal{N}:=\{ i=1,\dots,n \hbox{ s.t. } x_i \to x_i^*=0, \;  s_i \to s_i^*>0  \},
	\end{split}
\end{equation*}
{\edit where the optimal solution $(x^*,y^*,s^*)$ was defined in \eqref{eqn:optimal_solution_PPM}. Notice that $(\mathcal B,\mathcal N)$ is the partition corresponding to the optimal solution.}

\begin{assumption} \label{ass:asymptotic}
	We suppose that the following asymptotic estimates hold
	\begin{equation}\label{eq:variable_splitting_asympotitc}
		\begin{split}
			& s_i \in O(\mu) \hbox{ and } x_i \in O(1) \hbox{ for } i \in \mathcal{B}\\
			& x_i \in O(\mu)  \hbox{ and } s_i \in O(1) \hbox{ for } i \in \mathcal{N} \\
		\end{split}
	\end{equation}
	and, since $x^{-1}_is_i\approx \mu x^{-2}_i$ when an IPM iterate is sufficiently close to the central path, using  \eqref{eq:variable_splitting_asympotitc}, we suppose 
	\begin{equation*}
		\Theta_{ii}^{-1}=x^{-1}_is_i = O({\mu}) \hbox{ for }   i \in \mathcal{B}  \hbox{ and }  \; \; \Theta_{ii}^{-1}=x^{-1}_is_i= O(\mu^{-1}) \hbox{ for } i \in \mathcal{N}.  
	\end{equation*}
	
\end{assumption}
\noindent This assumption makes sense given the neighbourhood that is considered, see e.g.\ \cite{Gon:matrixfree}.

Due to Assumption \ref{ass:asymptotic}, we  consider the following asymptotic estimates of $(\Theta^{-1}+\rho I)^{-1}$
\begin{equation*} 
	(\Theta^{-1}+\rho I)^{-1}_{ii}=\begin{cases}
		O(\frac{1}{\rho+\mu}) & \hbox{ if  } i  \in \mathcal{B}\\
		O(\frac{\mu}{1+\rho \mu} ) & \hbox{ if } i \in \mathcal{N}.
	\end{cases}
\end{equation*}

{\edit The diagonal entries of $\Theta$ give a specific weight to each column of matrix $A$ (or equivalently, give a weight to each edge of the original sparse graph, as shown in Lemma~\ref{lem:Laplacian_Graph}). The columns for which the corresponding $\Theta_{ii}$ is $O(\mu)$ have a very small impact on the normal matrix, but still contribute to its sparsity pattern. In order to save time and memory when forming (complete or incomplete) Cholesky factorizations, we propose the following sparsification strategy: we introduce} a suitable threshold $C_t \in \mathbb{R}_+$  and  define

\begin{equation} \label{eq:theta_sparsification}
	(\widehat{\Theta}_{C_t \mu,\rho}^{\dagger})_{ii}:=\begin{cases}
		(\Theta^{-1}+\rho I)^{-1}_{ii} & \hbox{ if  } (\Theta^{-1}+\rho I)^{-1}_{ii} \geq \frac{C_t \mu}{1+\rho \mu}\\
		0 & \hbox{ if } (\Theta^{-1}+\rho I)^{-1}_{ii} < \frac{C_t \mu}{1+\rho \mu} .
	\end{cases}
\end{equation}
We define the $\mu$-sparisified version $S^{C_t \mu}_{\rho, \delta}$ of  $S_{\rho, \delta}$ as

\begin{equation}
\label{eqn:sparsified_normal_equations}
	S^{C_t \mu}_{\rho, \delta}:= A\widehat{\Theta}_{C_t \mu,\rho}^{\dagger}A^T+\delta I.
\end{equation}
{\edit Notice that this matrix completely ignores some of the columns of $A$ (and some of the edges of the graph). The dual regularization $\delta I$ guarantees that the resulting matrix is non-singular, irrespective of the level of sparsification chosen.} In this paper, we consider using inexact Newton directions produced by solving linear systems with matrix $S^{C_t \mu}_{\rho, \delta}$, rather than $S_{\rho, \delta}$.

\begin{remark}
It is important to note that, in general, the sparsity pattern of the matrix  $S^{C_t \mu}_{\rho, \delta}$ 	
depends on the choice of the parameter $C_t$ and on the partitioning $(\mathcal{B}, \mathcal{N})$. Indeed, when $\mu$ is sufficiently small, we expect that
	
	\begin{equation*}
	\Big|\Big\{i \in \{1, \dots, n\} \hbox{ s.t. }  (\Theta^{-1}+\rho I)^{-1}_{ii} \geq \frac{C_t \mu}{1+\rho \mu}\Big\}\Big|	= |\mathcal{B}| .
	\end{equation*}
	
%	Hence, the sparsity pattern  the matrix $S^{t \mu}_{\rho, \delta}$ depends on the partitioning $(\mathcal{B}, \mathcal{N})$.
	
\end{remark}

{\edit Let us now show how the algorithm is affected by the use of the proposed sparsified normal matrix. Notice that the results presented below depend strongly on two facts: the optimization problem evolves on a graph and thus the normal matrix is a Laplacian, with very desirable properties; the interior point method employs primal-dual regularization.}

{\edit We start by showing how much the normal matrix deviates from its sparsified counterpart.}

\begin{theorem} \label{th:difeerence_truncation}
The sparsification strategy in \eqref{eq:theta_sparsification} produces a matrix which is close to the original $S_{\rho,\delta}$, in the sense that
	\begin{equation*}
		\| S_{\rho, \delta}- S^{C_t \mu}_{\rho, \delta}\|=O\Big(\frac{C_t \mu}{1+\rho \mu}\Big).
	\end{equation*}
\end{theorem}

\begin{proof}
	We have that $S_{\rho, \delta}- S^{C_t \mu}_{\rho, \delta} = AE^{\mu}A^T$ where $E^{\mu}$ is the diagonal matrix defined as
	\begin{equation*}
		E^{\mu}_{ii}:= 	(\Theta^{-1}+\rho I)^{-1}_{ii}-(\widehat{\Theta}_{C_t\mu, \rho}^{\dagger})_{ii}=\begin{cases}
			0 & \hbox{ if } (\Theta^{-1}+\rho I)^{-1}_{ii} \geq \frac{C_t \mu}{1+\rho \mu}\\
			(\Theta^{-1}+\rho I)^{-1}_{ii}  & \hbox{ if } (\Theta^{-1}+\rho I)^{-1}_{ii} < \frac{C_t \mu}{1+\rho \mu}
		\end{cases}.
	\end{equation*}
Hence we have $\lambda_{max}(E^\mu) \leq \frac{C_t \mu}{1+\rho \mu}$. Thesis follows using Lemma \ref{lem:max_eig_lap} and observing that 
\begin{equation*}
	\frac{\mathbf{v}^T(S_{\rho, \delta}- S^{C_t \mu}_{\rho, \delta})\mathbf{v}}{\mathbf{v}^T\mathbf{v}}=\frac{\mathbf{v}^T(AE^\mu A^T)\mathbf{v}}{\mathbf{v}^T\mathbf{v}} \leq \lambda_{max}(E^\mu)\lambda_{max}(AA^T).
\end{equation*}
\end{proof}

{\edit We now show that the condition number of both matrices is uniformly bounded. This is an important property when using iterative Krylov solvers to find the Newton directions.}
\begin{lemma} \label{lem:conditioning_Schir_Complements}
	When $\mu$ is sufficiently small, the condition numbers of $S_{\rho, \delta}$ and $S^{C_t \mu}_{\rho, \delta}$ satisfy
	\begin{equation*}
		{k}_2( S_{\rho, \delta}) \in O\bigg(1+ \frac{1}{\delta(\rho +  \mu)}\bigg) \;\;\hbox{ and }\;\; {k}_2( S^{C_t \mu}_{\rho, \delta}) \in O\bigg(1+ \frac{1}{\delta(\rho +  \mu)}\bigg).
	\end{equation*}
\end{lemma}
\begin{proof}
	The thesis follows from Lemma \ref{lem:max_eig_lap} and observing that for $\mathbf v = A^T\mathbf w$	
	\begin{equation*}
		\begin{split}
			&\delta \leq \frac{\mathbf{w}^TS_{\rho, \delta} \mathbf{w} }{\mathbf{w}^T\mathbf{w} } \leq \delta + \frac{\mathbf{v}^T(\Theta^{-1}+\rho I)^{-1}\mathbf{v} }{\mathbf{v}^T\mathbf{v} }\frac{\mathbf{w}^TAA^T \mathbf{w} }{\mathbf{w}^T\mathbf{w} }  \\ 
			& \leq
			\delta + O\Big(\frac{1}{\rho+\mu} \big(2 \max_{v \in V} deg(v)\big)\Big) .
		\end{split}
	\end{equation*}
\end{proof}

{\edit We now show that the solution of the sparsified linear system is ``close" to the solution of the original one, and the bound depends on $\mu$. This result depends on the spectral distribution that was shown in the previous section.}

\begin{theorem}\label{theo:search_directions_discrepancy}
	For all $\mathbf{v} \in \mathbb{R}^m$ we have that
	\begin{equation*} \label{eq:search_directions_discrepancy}
		(S^{C_t\mu}_{\rho, \delta})^{-1}\mathbf{v}=S^{-1}_{\rho, \delta}\mathbf{v}+\boldsymbol{\boldsymbol{\psi}}
	\end{equation*}
where $\|\boldsymbol{\boldsymbol{\psi}}\| \in O(\frac{C_t \mu}{\delta^2 (1 + \rho \mu)}\|\mathbf v\|)$.
\end{theorem}

\begin{proof}
	Using Lemma \ref{lem:eigs_schur} and  Theorem \ref{th:difeerence_truncation},  we have that
	
	\begin{equation*}
		\|I-S_{\rho, \delta}^{-1}S_{\rho, \delta}^{C_t\mu}\|\leq \|S_{\rho, \delta}^{-1}\|\|  S_{\rho, \delta}- S^{C_t \mu}_{\rho, \delta} \| \in O\bigg(\frac{C_t\mu}{\delta(1+\rho \mu) }\bigg), 
	\end{equation*}
	and hence
	
	\begin{equation*}
		\|(S_{\rho, \delta}^{-1}-(S_{\rho, \delta}^{C_t\mu})^{-1})\mathbf{v}\|\leq \|(S_{\rho, \delta}^{C_t\mu})^{-1}\|\|S_{\rho, \delta}^{-1}\| \|  (S_{\rho, \delta}- S^{C_t \mu}_{\rho, \delta})\mathbf{v} \| \in O\bigg( \frac{C_t \mu}{\delta^2 (1 + \rho \mu)}\|\mathbf v\|\bigg).
	\end{equation*}
\end{proof}

{\edit
The next technical result is useful for the proof of Corollary~\ref{cor:final}.

\begin{lemma}
$\|\bar{\boldsymbol{\xi}}_p^j\|$ is uniformly bounded, i.e.\ there exists a constant $C_9>0$ such that for all $j\in\mathbb N$
\[\|\bar{\boldsymbol{\xi}}_p^j\|\le C_9
\]
\end{lemma}
\begin{proof}
For the sake of clarity of notation, we do not include the index $j$ in the proof.

To bound $\|\bar{\boldsymbol \xi}_p\|$, consider the following estimate
\[
\|\bar{\boldsymbol \xi}_p\| 
\le 
\|{\boldsymbol \xi}_p\| + \|A\|
\Big(
\|(\Theta^{-1}+\rho I)^{-1}X^{-1}{\boldsymbol \xi}_{\mu,\sigma}\|+
\|(\Theta^{-1}+\rho I)^{-1}\|\|{\boldsymbol \xi}_d\|
\Big).
\]
We already know the following estimates
\[
\|{\boldsymbol \xi}_p\|\le\frac{\mu n}{\gamma_p}\le \frac{\mu^0 n}{\gamma_p},\qquad
\|{\boldsymbol \xi}_d\|\le\frac{\mu n}{\gamma_d}\le \frac{\mu^0 n}{\gamma_d},\qquad
\|(\Theta^{-1}+\rho I)^{-1}\|\le\frac{1}{\rho}.
\]
To estimate $\|(\Theta^{-1}+\rho I)^{-1}X^{-1}{\boldsymbol \xi}_{\mu,\sigma}\|$, we proceed as in \eqref{eq:Delta_x_bound}:
\begin{align}
&\|(\Theta^{-1}+\rho I)^{-1}X^{-1}{\boldsymbol \xi}_{\mu,\sigma}\|=
\|(\Theta^{-1}+\rho I)^{-1}(S\mathbf e-\sigma\mu X^{-1}\mathbf e)\|\le\notag\\
&\le \|(\Theta^{-1}+\rho I)^{-1/2}\|\|(\Theta^{-1}+\rho I)^{-1/2}X^{-1/2}S^{1/2}\|\big(\|X^{1/2}S^{1/2}\mathbf e\|+\sigma\mu\|X^{-1/2}S^{-1/2}\mathbf e\|\big).\notag
\end{align}
It is straightforward to prove that
\[\|(\Theta^{-1}+\rho I)^{-1/2}\|\le\frac{1}{\rho^{1/2}},\qquad\|(\Theta^{-1}+\rho I)^{-1/2}X^{-1/2}S^{1/2}\|\le1.
\]
The remaining terms can be bounded using the properties of the neighbourhood
\[\|X^{1/2}S^{1/2}\mathbf e\| \le \sqrt{\mu\bar\gamma n},\quad
\sigma\mu\|X^{-1/2}S^{-1/2}\mathbf e\| \le \sigma\sqrt{\frac{\mu n}{\underline{\gamma}}}.
\]
Since $\mu\le\mu^0$, we deduce that $\|\bar{\boldsymbol{\xi}}_p\|\le C_9$, for some positive constant $C_9$.
\end{proof}
}

{\edit Finally, we show that, for a small enough constant $C_t$, the inexactness introduced by the sparsification strategy satisfies the Assumption \eqref{eq:inexact_assumption}. Therefore, an algorithm that includes such a sparsification strategy retains the polynomial complexity of the inexact IPM shown in Section \ref{sec:convergence}.}

\begin{corollary}
\label{cor:final}
	If in Algorithm \ref{alg:IPM} we generate the search directions using $(S^{C_t\mu}_{\rho, \delta})^{-1}$ with $C_t$ sufficiently small, i.e. if we compute the search directions using \eqref{eq:normal_system}, \eqref{eq:normal_sol_2_bis} and \eqref{eq:normal_sol_3_bis} where $S^{C_t \mu}_{\rho, \delta}$ substitutes $S_{\rho, \delta}$, then Algorithm \ref{alg:IPM} is convergent.
\end{corollary}
\begin{proof}
Using Theorem \ref{theo:search_directions_discrepancy}, we have

\begin{equation*} 
	(S^{C_t\mu}_{\rho, \delta})^{-1}\bar{\boldsymbol{\xi}}_p=S^{-1}_{\rho, \delta}\bar{\boldsymbol{\xi}}_p+\boldsymbol{\boldsymbol{\psi}}
\end{equation*}
where $\|\boldsymbol{\boldsymbol{\psi}}\| \leq C_{10} \frac{C_t \mu}{\delta^2 (1 + \rho \mu)}\|\bar{\boldsymbol \xi}_p\|$ for some constant $C_{10}>0$. Hence 
\begin{equation*} 
	S_{\rho, \delta} \underbrace{(S^{C_t\mu}_{\rho, \delta})^{-1}\bar{\boldsymbol{\xi}}_p}_{= \Delta \mathbf{y}}=\bar{\boldsymbol{\xi}}_p+S_{\rho, \delta}\boldsymbol{\boldsymbol{\psi}}.
\end{equation*}
Recall \eqref{eq:inexact_assumption} and Lemma \ref{lem:eigs_schur}; the thesis follows observing that there exists a constant $C_{11}>0$ s.t.
\begin{equation*}
\|  S_{\rho, \delta}\boldsymbol{\boldsymbol{\psi}} \| \leq 	C_{11} \frac{C_t \mu}{\delta^2 (1 + \rho \mu)} \leq C_\text{inexact} \mathbf{x}^T\mathbf{s}
\end{equation*}
where the last inequality holds if $$C_t <  \frac{\delta^2 (1 + \rho \mu) nC_\text{inexact}}{C_{11}}. $$
\end{proof}

\section{Numerical Results} \label{sec:numer_res}
The proposed method is compared with Lemon (Library for Efficient Modelling and Optimization on Networks) \cite{lemon_paper}, an extremely efficient \texttt{C++} implementation of the network simplex method, that has been shown to significantly outperform other popular implementations, like Cplex, see e.g.\ \cite{CasNas:ipm,ZanGon:OT}. The network simplex method has been shown \cite{SchSchGot:dotmark} to be very competitive against other algorithms specifically developed  for discrete OT, while remaining very robust and adaptable to many types of problems. Let us highlight that the other algorithms available in Lemon (cost scaling, capacity scaling, cycle cancelling) produced worse results than the network simplex.

All the computational tests discussed in this section are performed using a Dell PowerEdge R740 running Scientific Linux 7 with $4 \times$ Intel Gold 6234 3.3G, 8C/16T, 10.4GT/s, 24.75M Cache, Turbo, HT (130W) DDR4-2933, with 500GB of memory.
The PS-IPM implementation closely follows  the one from \cite{Cipolla_Gondzio} and is written in Matlab\textsuperscript{\textregistered}. The software versions used for the numerical experiments are as follows: Matlab R2022a, Lemon 1.3.1 and GCC 4.8.5 as the \texttt{C++} compiler. 

We stop Algorithm \ref{alg:PS-MF-IPM},  when
	\begin{equation}\label{eqn:Alg1_stop} 
	\|\mathbf{g} -A^T\mathbf{y}-\mathbf{s}  \|_{\infty} \leq \;R \cdot tol  \;\wedge\; {\|\mathbf{b} -A\mathbf{x} \|_1} \leq \;R \cdot tol \;\wedge\; C_{\mathbf{x}, \mathbf{s}} \leq \; tol,
\end{equation}
where  \begin{equation*}
	tol =10^{-10}, \, \, \, \,	R:=\max \{\|A\|_{\infty}, \|\mathbf{b}\|_1, \|\mathbf{c}\|_1 \},
\end{equation*}
and 
\begin{equation*}
	C_{\mathbf{x}, \mathbf{s}}:=\max_{i}\{\min\{|(\mathbf{x}_i\mathbf{s}_i)|,|\mathbf{x}_i|, |\mathbf{s}_i|\} \}.
\end{equation*}
Concerning the choice of the parameters in Algorithm \ref{alg:PS-MF-IPM}, we set $\sigma_r = 0.7$. Moreover, to prevent wasting time on finding excessively accurate solutions in the early PPM sub-problems, we set $\tau_1=10^{-4}$, i.e. we use as inexactness criterion for the PPM method
\begin{equation*}\label{eq:stopping_condition_empirical}
	\|\mathbf{r}_k(\mathbf{x}_{k+1},\mathbf{y}_{k+1}))\| < 10^4 \sigma_r^k \min\{1, \|(\mathbf{x}_{k+1}, \mathbf{y}_{k+1})-(\mathbf{x}_{k}, \mathbf{y}_{k}) \|.
\end{equation*}  
Indeed, in our computational experience, we have found that driving the IPM solver to a high accuracy in the initial PPM iterations is unnecessary and usually leads to a significant deterioration of the overall performance. 

Concerning Algorithm \ref{alg:IPM}, we set as regularization parameters $\rho=10^{-4}$ and $\delta = 10^{-6}$.
Moreover, in order to find the search direction, we employ a widely used predictor-corrector method \cite{MR1186163}. 
This issue represents the main point where practical implementation deviates from the theory in order to gain computational efficiency.

Finally, concerning the test problems, in all the following experiments we generate the load vector $\boldsymbol{\rho_1-\rho_0}$ in \eqref{eqn:problem_formulation} randomly and such that the sum of its entries is zero (to guarantee feasibility of the optimization problem), with only $10\%$ of them being nonzeros. Moreover, we fix the weight of each edge at $1$.

\subsection{Analysis of the sparsification strategy} \label{sec:sparsification_res}
In this section, we compare three possible solution strategies inside the PS-IPM: Cholesky  factorization  (using Matlab's \texttt{chol} function) applied to the full normal equations matrix \eqref{eqn:normal_equations_matrix}; Cholesky factorization (always using Matlab's \texttt{chol} function) applied to the sparsified matrix \eqref{eqn:sparsified_normal_equations}; preconditioned conjugate gradient (PCG) (using Matlab's \texttt{pcg} function) applied to the sparsified matrix \eqref{eqn:sparsified_normal_equations} with incomplete Cholesky preconditioner (computed using Matlab's \texttt{ichol} function). 
More in particular, as sparsification parameter in \eqref{eq:theta_sparsification} we use $C_t = 0.4$, $\texttt{`droptol'} = 10^{-3}$ in   \texttt{ichol} and $\texttt{`tol'} = 10^{-1}\mu$ in \texttt{pcg}.

We test the above mentioned solution strategies on various instances generated with the \texttt{CONTEST} generator \cite{Contest}; in particular, we considered the graphs  \texttt{pref}, \texttt{kleinberg}, \texttt{smallw} and \texttt{erdrey},  with a fixed number of $100,000$ nodes and different densities (i.e.\ average number of edges per node). Therefore, for these instances, $m=|V|=100,000$ and $n=|E|=m\cdot\text{density}$.

In the upper panels of Figures~\ref{fig:sparsification_analysis} and \ref{fig:sparsification_analysis2} we report the computational time of the three approaches for various values of densities (chosen in relation to the properties of the graph), whereas in the lower panels we report the total number of IPM iterations.
From the presented numerical results, it is clear that the sparsification strategy, in conjunction with the iterative solution of the linear systems, provides a clear advantage over the use of a direct factorization. As can be expected, the iterative method and the sparsification strategy become more advantageous when the size of the problem (number of edges) increases. On the other hand, it is important to note that the use of the sparsified Newton equations in conjunction with the full Cholesky factorization presents only limited advantages in terms of computational time when compared to the Cholesky factorization of the full Newton normal equation. This is the case because the resulting inexact IPM requires, generally, more iterations to converge (see lower panels of Figures~\ref{fig:sparsification_analysis} and \ref{fig:sparsification_analysis2}). Advantages of the proposed approach become clearer when the graphs are denser.

\begin{figure}[h]
\centering
\caption{Comparison of sparsified and full normal equations approach, using full Cholesky factorization or incomplete Cholesky as preconditioner for PCG, in terms of IPM iterations and computational time, {\edit for problems {\tt pref} and {\tt kleinberg}.}}
\label{fig:sparsification_analysis}
\includegraphics[width=\textwidth]{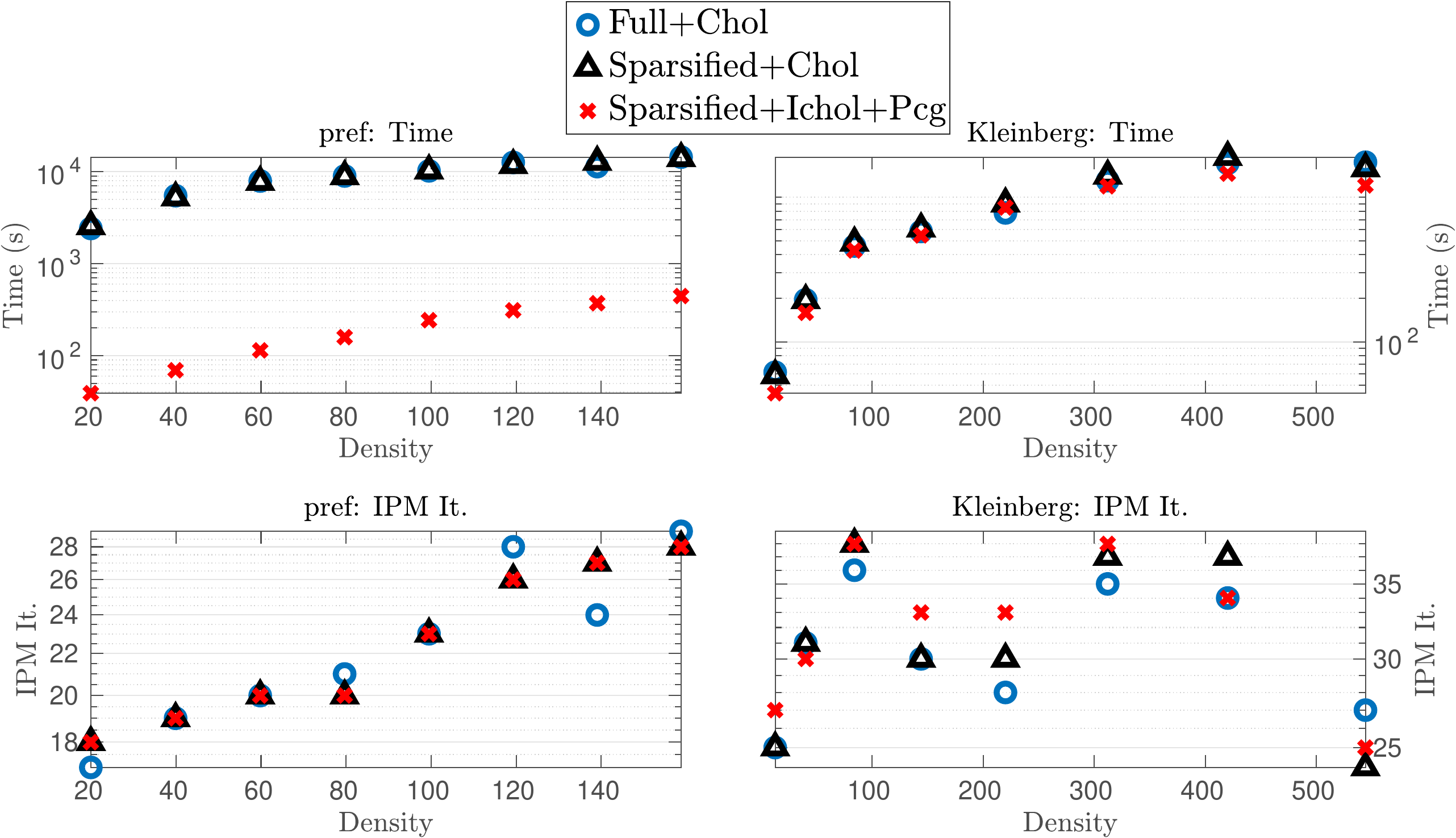}
\end{figure}

\begin{figure}[h]
	\centering
	\caption{Comparison of sparsified and full normal equations approach, using full Cholesky factorization or incomplete Cholesky as preconditioner for PCG, in terms of IPM iterations and computational time, {\edit for problems {\tt erdrey} and {\tt smallw}.}}
	\label{fig:sparsification_analysis2}
	\includegraphics[width=\textwidth]{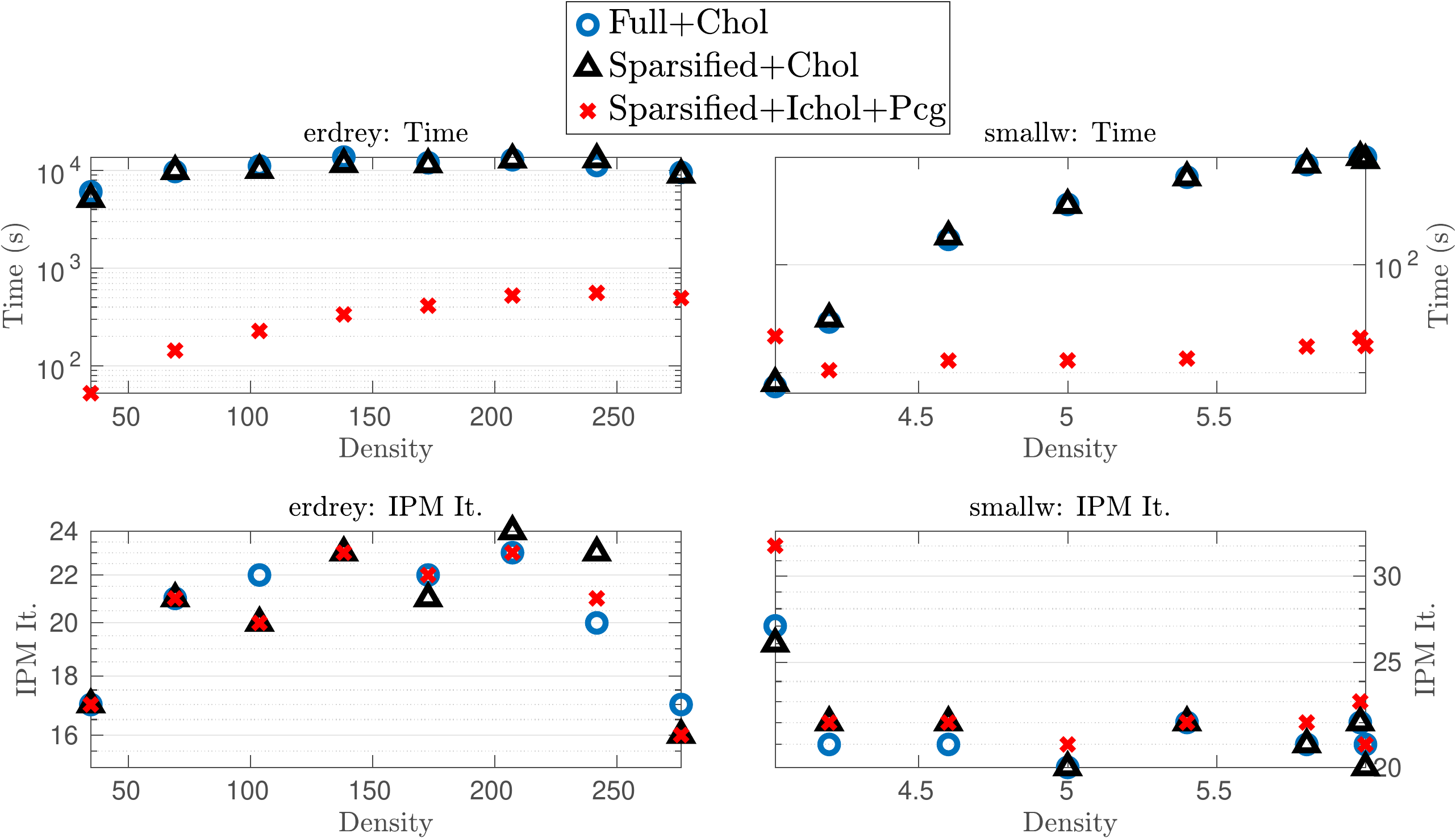}
\end{figure}

\subsection{Results on randomly generated graphs}
In this section, we compare the PS-IPM algorithm, using the sparsified normal equations matrix and the PCG, with the network simplex solver of Lemon. For PS-IPM we use the same parameters as proposed in Section \ref{sec:sparsification_res}. The graphs used in this section come from the generator developed in \cite{VigLat:generator} and already used for OT on graphs in \cite{MR3820384}. This generator produces random connected graphs with a number of nodes varying from $1,000$ to $10,000,000$ and degrees of each node in the range $[1,10]$, with an average of $5$ edges per node. For each size, $10$ graphs and load vectors are generated and tested. These parameters closely resemble the ones used in \cite{MR3820384}.

Figure~\ref{fig:results_random_generator1} shows the comparison of the computational time between PS-IPM and Lemon: for each size of the problem (indicated by the total number of edges), we report the summary statistics of the execution times using Matlab's \texttt{boxplot}.

\begin{figure}[h]
	\centering
	\caption{Box Plots of the computational times of PS-IPM and Lemon, for randomly generated graphs. The red and black intervals show the spread of the measured computational times (red, on the left, is PS-IPM; black, on the right, is Lemon), while the (blue) crosses indicate the outliers.}
	\label{fig:results_random_generator1}
	\includegraphics[width=\textwidth]{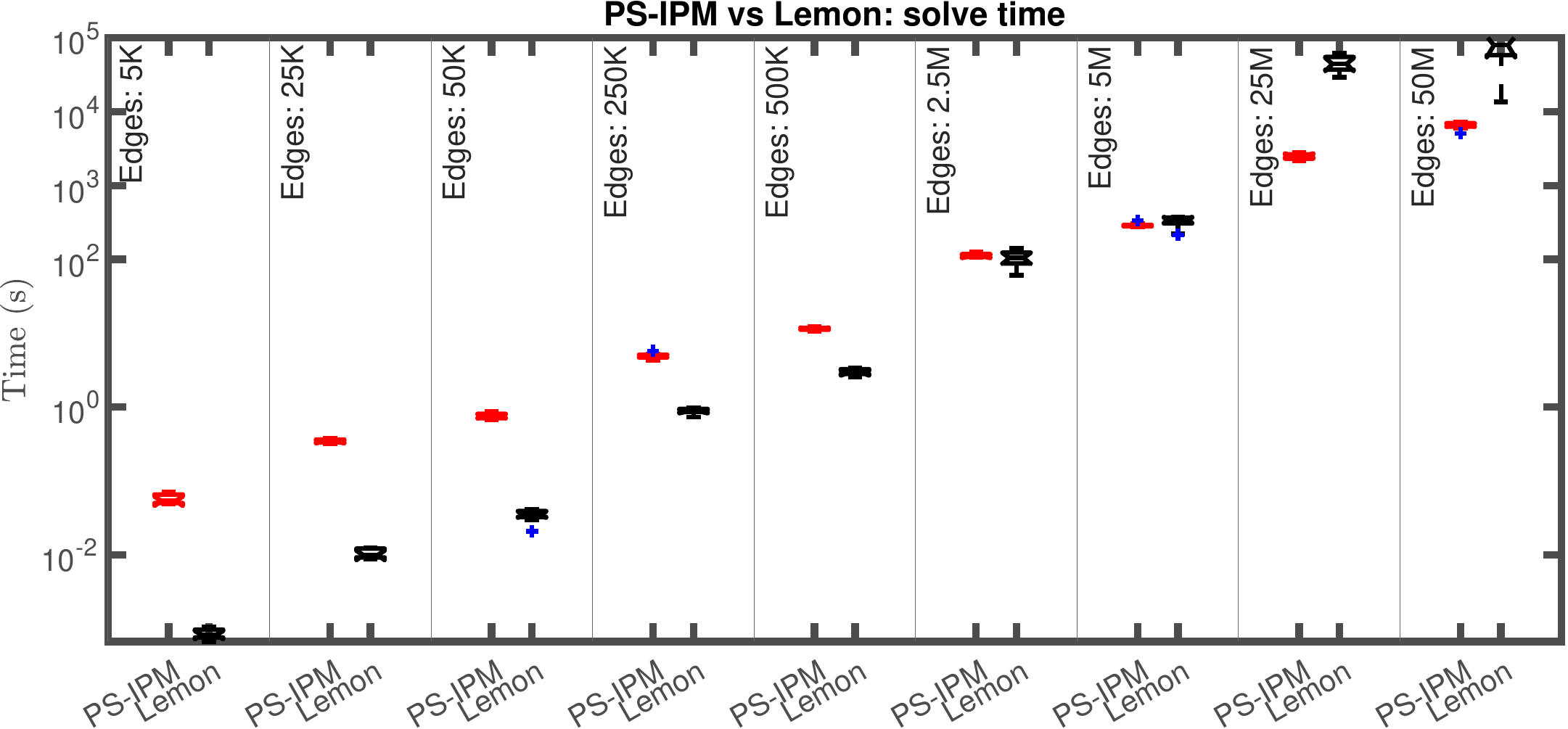}
\end{figure} 

For small size problems, Lemon is the clear winner, by two orders of magnitude; however, as the size increases, the performance difference between the two methods reduces and for the largest instance considered, Lemon becomes one order of magnitude slower than PS-IPM.

{\edit Figure~\ref{fig:results_random_generator} shows the average computational time against the number of edges (from $5,000$ to $50M$) in a logarithmic scale, the corresponding regression lines and their slopes.  Both the proposed method and the network simplex (see \cite{Orl:polynomial}) are known to have polynomial complexity in terms of number of iterations (although the estimates are usually very pessimistic for IPMs); from the computational results presented,  we can estimate the practical time complexity of both methods. Recall that, in a log-log plot, polynomials of the type $x^m$ appear as straight lines with slope $m$. Using linear regression, we can estimate that the time taken by Lemon grows with exponent approximately $2.06$, while the time taken by PS-IPM grows with exponent approximately $1.28$, providing a considerable advantage for large sizes.
}

{\edit Looking at the full set of results as reported in Figure~\ref{fig:results_random_generator1}, let us mention finally the fact that the variance of the computational times over the 10 runs for a given problem size is smaller when using PS-IPM, especially for large sizes,} indicating that the method is more robust and less dependent on the specific problem being solved. This is a very desirable property.

\begin{figure}[t]
\centering
\caption{Logarithmic plot of the computational time for randomly generated graphs. The time taken by Lemon (red circles) grows as $(\text{number of edges})^{2.06}$; the time taken by PS-IPM (blue triangles) grows as $(\text{number of edges})^{1.28}$.}
\label{fig:results_random_generator}
\includegraphics[width=\textwidth]{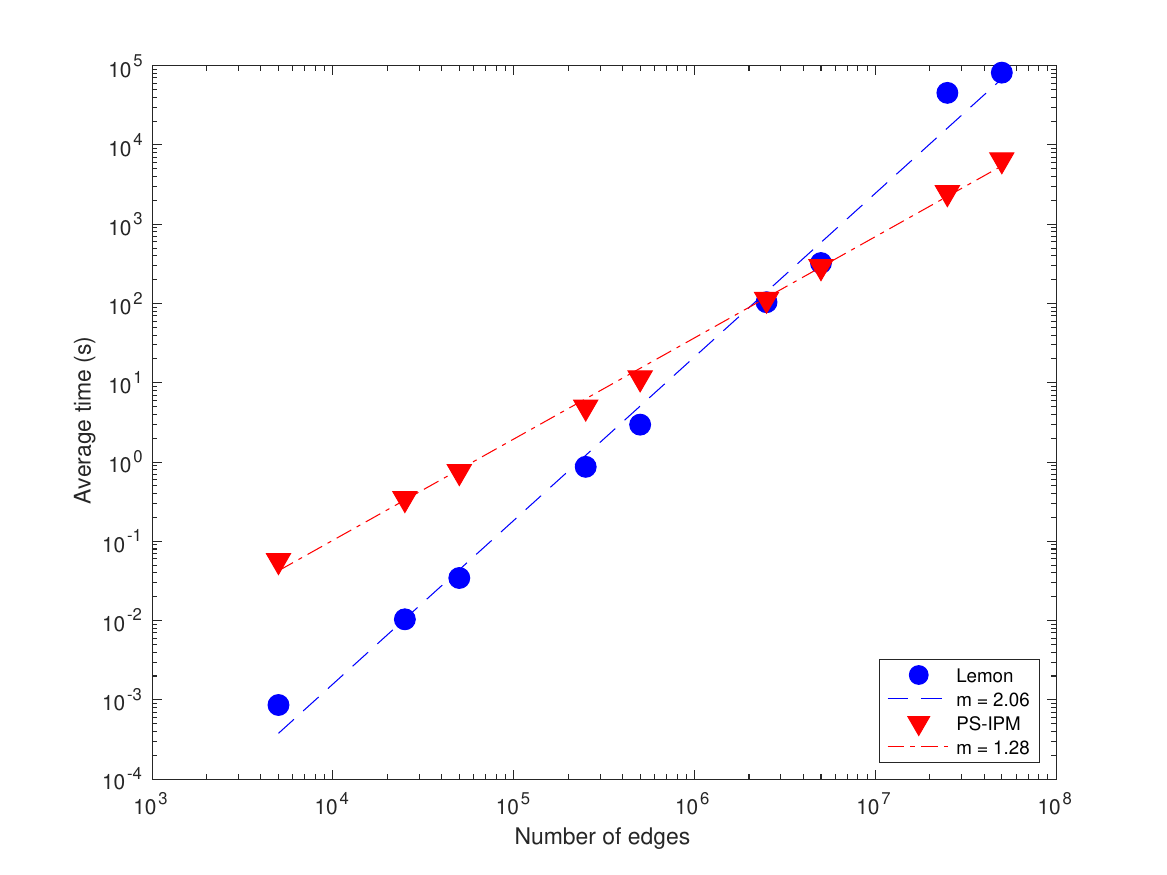}
\end{figure}

%\begin{figure}[h]
%\centering
%\caption{Box Plots of the computational times of PS-IPM and Lemon, for randomly generated graphs. The red and black intervals show the spread of the measured computational times (red, on the left, is PS-IPM; black, on the right, is Lemon), while the (blue) crosses indicate the outliers.}
%\label{fig:results_random_generator}
%	\includegraphics[width=\textwidth]{./Ps_vs_Lem}
%\end{figure} 

\subsection{Results on SuiteSparse graphs}
{\edit
Results on randomly generated problems do not necessarily represent the ability of an optimization method to tackle problems coming from real world applications. Therefore,}
in this section, we show the results of applying PS-IPM and Lemon to some sparse graphs from the SuiteSparse matrix collection \cite{MR2865011}. The characteristics of the graphs considered are shown in Table~\ref{tab:suitesparse}:  the number of nodes, edges and the average number of edges per node. All the graphs are undirected and connected. Due to the fact that the considered graphs are particularly sparse, in the numerical results presented in this section, we solve the sparsified normal equations using the full Cholesky factorization.

\begin{table}[h]
\caption{Details of the graphs from the SuiteSparse matrix collection, ordered by increasing number of edges.}
\label{tab:suitesparse}
\centering
\begin{tabular}{rrrr}
\hline
Name & Nodes & Edges & Density\\
\hline
\texttt{nc2010} & 288,987 & 1,416,620 & 4.9\\
\texttt{NACA0015} & 1,039,183 & 6,229,636 & 6.0\\
\texttt{great-britain-osm} & 7,733,822 & 16,313,034 & 2.1\\
\texttt{hugetric-00010} & 6,592,765 & 19,771,708 & 3.0\\
\texttt{hugetric-00020} & 7,122,792 & 21,361,554 & 3.0\\
\texttt{hugetrace-00010} & 12,057,441 & 36,164,358 & 3.0\\
\texttt{hugetrace-00020} & 16,002,413 & 47,997,626 & 3.0\\
\texttt{delaunay-n23} & 8,388,608 & 50,331,568 & 6.0\\
\hline
\end{tabular}
\end{table}

Figure~\ref{fig:results_SuiteSparse_logplot} shows the computational times for the eight problems considered, using PS-IPM and Lemon. Apart from the problem \texttt{nc2010}, which represents a relatively small instance in out dataset, on all the other problems PS-IPM consistently outperforms Lemon in terms of required computational time. In particular, for the problems \texttt{hugetric-00010}, \texttt{hugetric-00020}, \texttt{hugetrace-00010} and \texttt{hugetrace-00020}, which reach up to 16 million nodes and 48 million edges, PS-IPM is one order of magnitude faster than Lemon.

{\edit
Notice that graphs of these sizes (and larger) appear in many modern practical applications, e.g. social networks, PageRank, analysis of rail/road networks, energy models, to mention a few.

Looking at the regression lines and their slopes, we notice that the time taken by Lemon grows with exponent approximately $2.07$ while the time taken by PS-IPM grows with exponent approximately $1.40$. These values are very close to the ones found previously for randomly generated graphs. The data of Figure \ref{fig:results_SuiteSparse_logplot} however has a more erratic behaviour than the times shown in Figure \ref{fig:results_random_generator}, because the properties of each graph considered are different and because we are not averaging over $10$ different instances of each problem.

Let us highlight also that the time taken by Lemon seems to be more problem dependent, while PS-IPM looks more consistent and robust.
}

%\begin{figure}[h]
%	\centering
%	\caption{Computational results for the SuiteSparse problems: the red dots show the computational time for PS-IPM, the blue stars for Lemon.}
%	\label{fig:results_SuiteSparse}
%	\includegraphics[width=\textwidth]{./suitesparse}
%\end{figure} 

\begin{figure}[t]
\centering
\caption{Logarithmic plot of the computational time for the SuiteSparse problems. The time taken by Lemon (red circles) grows as $(\text{number of edges})^{2.07}$; the time taken by PS-IPM (blue triangles) grows as $(\text{number of edges})^{1.40}$. The instances are ordered as in Table \ref{tab:suitesparse}.}
\label{fig:results_SuiteSparse_logplot}
\includegraphics[width=\textwidth]{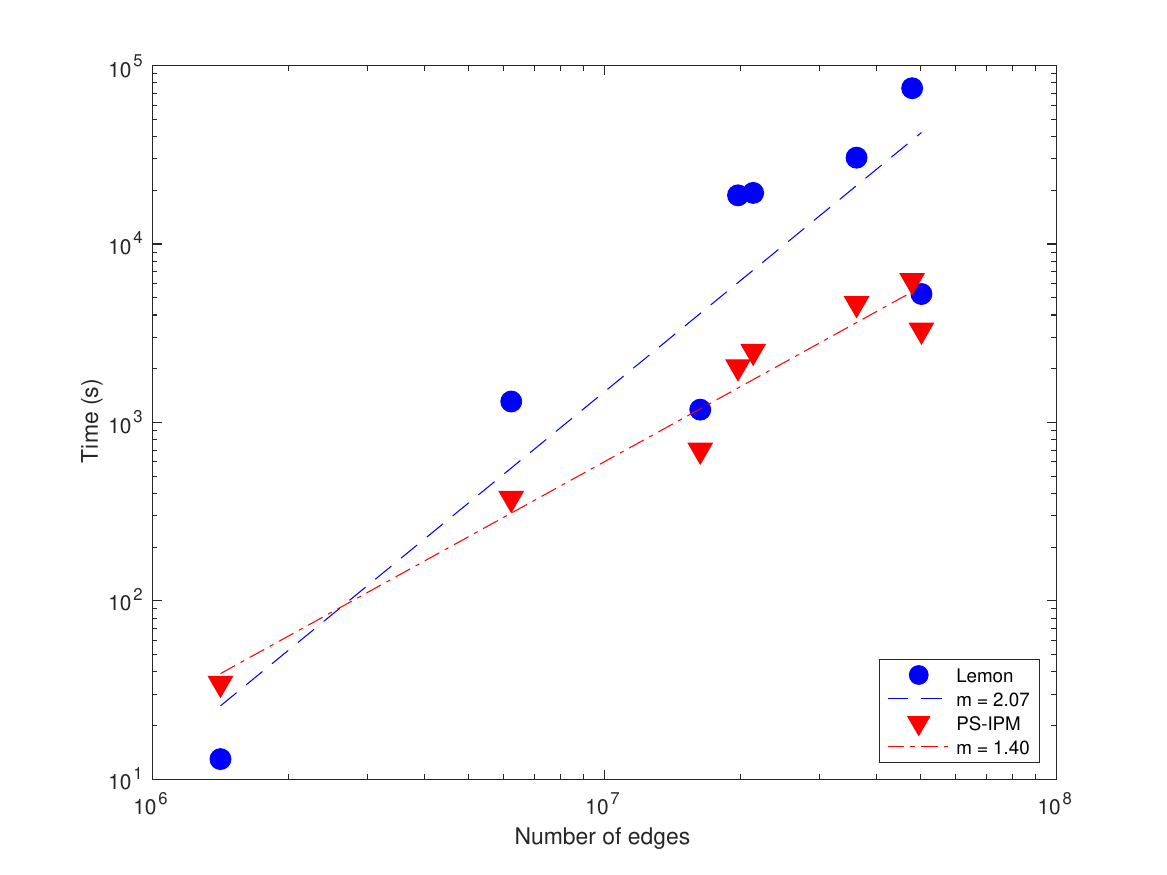}
\end{figure}

\section{Conclusion}
An efficient computational framework for the solution of Optimal Transport problems on graphs has been presented in this paper. Such framework relies on Proximal-Stabilized Interior Point Method and clever sparsifications of the normal Newton equations to compute the inexact search directions. The proposed technique is sound and polynomial convergence guarantee has been established for the inner inexact IPM. Extensive numerical experiments show that for large scale problems, a simple prototype \texttt{Matlab} implementation is able to outperform consistently  a highly specialized and very efficient \texttt{C++} 
implementation of the network simplex method. 

{\edit We highlight also that Interior Point Methods are more easily parallelizable than simplex-like methods; for huge scale problems, for which high performance computing resources need to be used, the use of IPMs with proper parallelization may be the only viable strategy to solve these problems.}

\bibliographystyle{siam}
\bibliography{OT_bib}

\end{document}